\documentclass[10pt]{amsart}
\usepackage{amscd}
\usepackage{amsmath}
\usepackage{amsfonts}
\usepackage{amssymb}
\usepackage{amsthm}
\usepackage{enumerate}
\usepackage[T1]{fontenc}
\usepackage{hyperref}
\usepackage{mathrsfs}
\usepackage{stackrel}
\usepackage[all]{xy}
\usepackage{yhmath}

\DeclareMathOperator{\gr}{gr}
\DeclareMathOperator{\Gr}{Gr}
\DeclareMathOperator{\Ext}{Ext}
\DeclareMathOperator{\Hom}{Hom}
\DeclareMathOperator{\sExt}{\mathscr{E}\mathit{xt}}
\DeclareMathOperator{\sHom}{\mathscr{H}\mathit{om}}

\DeclareMathOperator{\Der}{Der}

\DeclareMathOperator{\Sp}{Sp}
\DeclareMathOperator{\Spf}{Spf}
\DeclareMathOperator{\Spec}{Spec}
\DeclareMathOperator{\Sym}{Sym}

\DeclareMathOperator{\rig}{rig}

\DeclareMathOperator{\Loc}{Loc}

\DeclareMathOperator{\id}{id}

\DeclareMathOperator{\Ann}{Ann}
\begin{document}
\newtheorem{MainThm}{Theorem}
\renewcommand{\theMainThm}{\Alph{MainThm}}
\theoremstyle{definition}
\newtheorem*{examps}{Examples}
\newtheorem*{defn}{Definition}
\theoremstyle{plain}
\newtheorem*{lem}{Lemma}
\newtheorem*{prop}{Proposition}
\newtheorem*{rmk}{Remark}
\newtheorem*{rmks}{Remarks}
\newtheorem*{thm}{Theorem}
\newtheorem*{example}{Example}
\newtheorem*{examples}{Examples}
\newtheorem*{cor}{Corollary}
\newtheorem*{conj}{Conjecture}
\newtheorem*{hyp}{Hypothesis}
\newtheorem*{thrm}{Theorem}
\newtheorem*{quest}{Question}
\theoremstyle{remark}

\newcommand{\Zp}{{\mathbb{Z}_p}}
\newcommand{\Qp}{{\mathbb{Q}_p}}
\newcommand{\Fp}{{\mathbb{F}_p}}
\newcommand{\A}{\mathcal{A}}
\newcommand{\B}{\mathcal{B}}
\newcommand{\C}{\mathcal{C}}
\newcommand{\D}{\mathcal{D}}
\newcommand{\E}{\mathcal{E}}
\newcommand{\F}{\mathcal{F}}
\newcommand{\G}{\mathcal{G}}
\newcommand{\I}{\mathcal{I}}
\newcommand{\J}{\mathcal{J}}
\renewcommand{\L}{\mathcal{L}}
\newcommand{\sL}{\mathscr{L}}
\newcommand{\M}{\mathcal{M}}
\newcommand{\sM}{\mathscr{M}}
\newcommand{\N}{\mathcal{N}}
\newcommand{\sN}{\mathscr{N}}
\renewcommand{\O}{\mathcal{O}}
\newcommand{\cP}{\mathcal{P}}
\newcommand{\Q}{\mathcal{Q}}
\newcommand{\R}{\mathcal{R}}
\newcommand{\cS}{\mathcal{S}}
\newcommand{\T}{\mathcal{T}}
\newcommand{\U}{\mathcal{U}}
\newcommand{\sU}{\mathscr{U}}
\newcommand{\sV}{\mathscr{V}}
\newcommand{\V}{\mathcal{V}}
\newcommand{\W}{\mathcal{W}}
\newcommand{\X}{\mathcal{X}}
\newcommand{\Y}{\mathcal{Y}}
\newcommand{\Z}{\mathcal{Z}}
\newcommand{\PFS}{\mathbf{PreFS}}
\newcommand{\h}[1]{\widehat{#1}}
\newcommand{\hA}{\h{A}}
\newcommand{\hK}[1]{\h{#1_K}}
\newcommand{\hsULK}{\hK{\sU(\L)}}
\newcommand{\hsUFK}{\hK{\sU(\F)}}
\newcommand{\hsULnK}{\hK{\sU(\pi^n\L)}}
\newcommand{\hn}[1]{\h{#1_n}}
\newcommand{\hnK}[1]{\h{#1_{n,K}}}
\newcommand{\w}[1]{\wideparen{#1}}
\newcommand{\wK}[1]{\wideparen{#1_K}}
\newcommand{\invlim}{\varprojlim}
\newcommand{\dirlim}{\varinjlim}
\newcommand{\fr}[1]{\mathfrak{{#1}}}
\newcommand{\LRU}[1]{\sU(\mathscr{#1})}
\newcommand{\et}{\acute et}

\newcommand{\ts}[1]{\texorpdfstring{$#1$}{}}
\newcommand{\st}{\mid}
\newcommand{\be}{\begin{enumerate}[{(}a{)}]}
\newcommand{\ee}{\end{enumerate}}
\newcommand{\qmb}[1]{\quad\mbox{#1}\quad}
\let\le=\leqslant  \let\leq=\leqslant
\let\ge=\geqslant  \let\geq=\geqslant

\title[$\w{\D}$-modules on rigid analytic spaces III]{$\w\D$-modules on rigid analytic spaces III: \\Weak holonomicity and operations}

\author{Konstantin Ardakov}
\address{Mathematical Institute\\University of Oxford\\Oxford OX2 6GG}
\author{Andreas Bode}
\author{Simon Wadsley}
\address{Homerton College, Cambridge, CB2 8PQ}

\maketitle
\begin{abstract}
We develop a dimension theory for coadmissible $\w{\D}$-modules on rigid analytic spaces and study those which are of minimal dimension, in analogy to the theory of holonomic $\D$-modules in the algebraic setting. We discuss a number of pathologies contained in this subcategory (modules of infinite length, infinte-dimensional fibres). \\
We prove stability results for closed immersions and the duality functor, and show that all higher direct images of integrable connections restricted to a Zariski open subspace are coadmissible of minimal dimension. It follows that the local cohomology sheaves $\underline{H}^i_Z(\M)$ with support in a closed analytic subset $Z$ of $X$ are also coadmissible of minimal dimension for any integrable connection $\M$ on $X$.
\end{abstract}
\section{Introduction}
Let $K$ be a complete discrete valuation field of characteristic zero with valuation ring $R$ and uniformiser $\pi\in R$. We allow both the case of mixed characteristic (e.g. finite field extensions of the $p$-adic numbers $\mathbb{Q}_p$) and equal characteristic (e.g. $\mathbb{C}((t))$).\\
In \cite{DhatOne}, the first and the third author introduced the sheaf $\w{\D}_X$ of analytic (infinite order) differential operators on a smooth rigid analytic $K$-space $X$. It was shown in \cite{DhatOne} and in greater generality in \cite{Bode} that sections over affinoids are \emph{Fr\'echet--Stein algebras} as defined by Schneider--Teitelbaum \cite{ST}, which suggests the notion of \emph{coadmissibility} as the natural analogue of coherence in this setting. We denote the category of coadmissible $\w{\D}_X$-modules by $\C_X$.\\
\\
In the classical theory of $\D$-modules (on a smooth complex algebraic variety $X$, say), one is often particularly interested in those modules which are \emph{holonomic}. There are various equivalent ways to define these, one of which is the following: one can introduce a dimension function $d$ for coherent $\D$-modules, either as the dimension of the support of the associated characteristic variety, or in terms of homological algebra by interpreting the homological grade of a module as its codimension. One then shows that any non-zero coherent $\D_X$-module $\M$ satisfies
\begin{equation*}
\dim X\leq d(\M)\leq 2\dim X,
\end{equation*}
which is known as \emph{Bernstein's inequality}. A coherent $\D_X$-module $\M$ is said to be holonomic if $d(\M)\leq \dim X$, i.e. $\M$ is either zero or of minimal dimension.\\
\\
Equivalently, holonomic $\D_X$-modules can be characterized as those coherent $\D_X$-modules $\M$ satisfying either of the following equivalent properties:\\
\\
 $(*)$ for every point $i: x\to X$ and any $j\in \mathbb{Z}$, the cohomology $\mathrm{H}^j(i^+\M)$ is a finite-dimensional vector space over $\mathbb{C}$, where $i^+$ denotes the derived inverse image of $\D$-modules (see \cite[Theorem 3.3.1]{HTT}).\\
\\
$(**)$ for any smooth morphism $f: X'\to X$ and any divisor $Z$ of $X'$, the local cohomology sheaves $\underline{H}^i_Z(f^*\M)$ are coherent $\D_{X'}$-modules for any $i\geq 0$ (see property $(**)$ in the introduction of \cite{Caro04}).\\
\\
The category of holonomic $\D_X$-modules plays a crucial role in many parts of algebraic geometry, algebraic analysis and geometric representation theory. It contains all integrable connections on $X$, and each holonomic module has finite length. Moreover, the notion of holonomicity is stable under pullback, pushforward, tensor product and the duality functor, which takes the form
\begin{equation*}
\mathbb{D}: \M\mapsto \sExt^{\dim X}_{\D_X}(\M, \D_X)\otimes_{\O_X}\Omega^{\otimes -1}
\end{equation*} 
for any holonomic $\D_X$-module $\M$.\\
Moreover, the Riemann--Hilbert correspondence asserts an equivalence of categories between holonomic modules with regular singularities and the category of perverse sheaves.\\
\\
In this paper, we begin the study of a subcategory of $\C_X$ analogous to the category of holonomic $\D$-modules. While there is currently no satisfactory theory of characteristic varieties for coadmissible $\w{\D}$-modules, we can adopt the homological viewpoint by slightly generalising the dimension theory for Fr\'echet--Stein algebras given in \cite{ST}, where results were given for Fr\'echet--Stein algebras defined by Banach algebras which are Auslander regular with universally bounded global dimension. We relax this condition by allowing Banach algebras which are Auslander--Gorenstein with universally bounded self-injective dimension. This allows us to define the dimension of a coadmissible $\w{\D}_X$-module. We then prove the corresponding Bernstein inequality.
\begin{MainThm}\label{introthmA}
Let $X$ be a smooth affinoid $K$-space such that $\T(X)$ is a free $\O(X)$-module. 
\begin{enumerate}[(i)]
\item There is a Fr\'echet--Stein structure $\w{\D}(X)\cong \varprojlim A_n$, where each $A_n$ is Auslander--Gorenstein with self-injective dimension bounded by $2\dim X$.
\item If $M$ is a non-zero coadmissible $\w{\D}(X)$-module then 
\begin{equation*}
d(M)\geq \dim X,
\end{equation*}
where $d(M)=2\dim X-j(M)$ for $j(M)$ the homological grade of $M$.
\end{enumerate}
\end{MainThm}
Concerning (i), it is worth pointing out that our computation of Ext groups shows that the $A_n$ have in fact self-injective dimension exactly $\dim X$ for $n$ sufficiently large, but we won't use this fact.\\ 
We also note that Mebkhout--Narvaez-Macarro have already discussed dimensions of modules over the sheaf $\D$ of \emph{algebraic} (i.e. finite order) differential operators on a rigid analytic space in \cite{Mebkhout}, and we show that the two theories are compatible in the obvious way. This rests on the following theorem.
\begin{MainThm}
\label{introthmC}
Let $X$ be a smooth affinoid $K$-space. Then $\w{\D}(X)$ is a faithfully flat $\D(X)$-module.
\end{MainThm}

We call a coadmissible $\w{\D}$-module $\M$ on $X$ which satisfies $d(\M)\leq \dim X$ \emph{weakly holonomic}. \\
This choice of nomenclature reflects the fact that the category of weakly holonomic modules still contains some pathologies which do not appear in the algebraic theory: we present examples of weakly holonomic modules which are not of finite length and have infinite-dimensional fibres. In particular, the natural analogues of $(*)$ and $(**)$ do not provide equivalent characterizations of weak holonomicity.\\
In \cite{Bitoun}, Bitoun and the second author already gave an example of an integrable connection on the punctured unit disc such that its direct image on the disc is not even coadmissible, so weakly holonomic $\w{\D}$-modules are not stable under pushforward either. \\
\\
Nonetheless, we also produce some positive results. One can define a duality functor as in the classical setting and show that this gives an involution of the category of weakly holonomic modules. We also show that the $\w{\D}$-module analogue of Kashiwara's equivalence given in \cite{DhatTwo} respects weak holonomicity. Concerning the question of pushforwards along open embeddings, we prove the following.
\begin{MainThm}
\label{introthmB}
Let $j: U\to X$ be a Zariski open embedding of smooth rigid analytic $K$-spaces and let $\M$ be an integrable connection on $X$. Then $\mathrm{R}^ij_*(\M|_U)$ is a coadmissible, weakly holonomic $\w{\D}_X$-module for any $i\geq 0$.
\end{MainThm}
By `Zariski open embedding', we mean that $U$ is an admissible open subspace of $X$ whose complement is a closed analytic subset of $X$.\\
The proof of Theorem \ref{introthmB} relies on the rigid analytic analogue of Hironaka's resolution of embedded singularities as developed by Temkin \cite{Temkin}.\\
As a corollary, we obtain that the local cohomology sheaves $\underline{H}^i_Z(\M)$ are also coadmissible, weakly holonomic $\w{\D}_X$-modules, where $Z$ is any closed analytic subset of $X$. \\
In this way, we verify that any integrable connection $\M$ on $X$ satisfies the following natural analogue of $(**)$:\\
\\
$(**')$ for any smooth morphism $f: X'\to X$ and any divisor $Z$ of $X'$, the local cohomology sheaves $\underline{H}^i_Z(f^*\M)$ are coadmissible for any $i\geq 0$.\\
\\
We mention at this point that Caro has taken property $(**)$ as the point of departure for his development of the study of \emph{overcoherent} arithmetic $\mathscr{D}$-modules \cite{Caro04}. It would be very interesting to investigate whether $(**')$ (maybe together with some analogue of Caro's \emph{overholonomicity} condition \cite{Caro09}) yields a sufficiently rich subcategory of weakly holonomic $\w{\D}$-modules which has better finiteness and stability properties.\\
\\
We hope that these results bring us closer to the formulation of a theoretical framework which allows for a $p$-adic Riemann--Hilbert correspondence for $\w{\D}_X$-modules, generalizing results by Liu-Zhu \cite{LZ} for (a suitable category of) integrable connections and (de Rham) local systems. 
\subsection*{Structure of the paper}
In section 2, we recall some of the results and terminology from \cite{DhatOne}.\\
\\
In section 3, we prove Theorem \ref{introthmC}.\\
\\
In section 4, we show that the sections $\w{\D}(X)$ over a smooth affinoid $X$ are of the form as claimed in Theorem \ref{introthmA}.(i). In section 5, we slightly generalise the dimension theory from \cite{ST} to algebras of this form.\\
\\
In section 6, we prove Bernstein's inequality, Theorem \ref{introthmA}.(ii).\\
\\
In section 7, we show some basic properties of the category of weakly holonomic $\w{\D}$-modules. Amongst other things, we prove that every integrable connection is weakly holonomic, and discusss the duality functor.\\
\\
In section 8, we present examples which have no analogue in the classical theory: there exist weakly holonomic $\w{\D}$-modules which do not have finite length and have fibres of infinite dimension. We also briefly recall the results from \cite{Bitoun}, which is concerned with the extension of meromorphic connections to coadmissible $\w{\D}$-modules. Crucially, it also gives an example of an integrable connection whose direct image is not coadmissible.\\
\\
In sections 9 and 10, we prove Theorem \ref{introthmB} by first considering the case where the complement of $U$ is the analytification of a strict normal crossing divisor and then reducing to that case by invoking Temkin's resolution of embedded singularities \cite{Temkin} and the results in \cite{Bodeproper}.

\subsection*{Convention} All our rigid analytic spaces will be quasi-separated.\\
Throughout, smooth rigid analytic spaces will be assumed to be equidimensional for simplicity. Arguing on each connected component separately, analogues of all our results can be formulated for arbitrary smooth spaces in a straightforward manner.
\subsection*{Notation} Given an $R$-module $M$, we denote by $\h{M}$ its $\pi$-adic completion and abbreviate $\h{M}\otimes_R K$ to $\hK{M}$.\\
\\
The first and the second author acknowledge support from the EPSRC grant EP/L005190/1.

\section{Basic theory of $\w{\D}$-modules}
We recall some definitions and results from \cite{DhatOne}.

\subsection{Fr\'echet completed enveloping algebras}
Let $k$ be a commutative base ring and $A$ a commutative $k$-algebra. \\
A \emph{Lie--Rinehart algebra} (or $(k, A)$-Lie algebra) is an $A$-module $L$ equipped with a $k$-bilinear Lie bracket and an \emph{anchor map}
\begin{equation*}
\rho: L\to \Der_k(A)
\end{equation*}
such that $[x, ay]=a[x, y]+\rho(x)(a)y$ for any $a\in A$, $x, y\in L$.\\
\\
We say that $L$ is \emph{smooth} if it is coherent and projective as an $A$-module.\\
\\
For any $(k, A)$-Lie algebra $L$, one can form the enveloping algebra $U_A(L)$ as in \cite[\textsection 2]{Rinehart}. This is an associative $k$-algebra with the property that to give a $U_A(L)$-module structure on an $A$-module $M$ is equivalent to giving a Lie algebra action of $L$ on $M$ such that 
\begin{equation*}
x\cdot (a\cdot m)=(ax)\cdot m+\rho(x)(a)\cdot m \ \text{and}\  a\cdot (x\cdot m)=(ax)\cdot m
\end{equation*}
for any $a\in A$, $x\in L$, $m\in M$. \\
\\
If $A$ is an affinoid $K$-algebra, we say that an $R$-subalgebra $\A\subset A$ is an \emph{affine formal model} if $\A\otimes_R K=A$ and $\A$ is a topologically finitely presented $R$-algebra.\\
\\
A finitely generated $\A$-submodule $\L$ of a $(K, A)$-Lie algebra $L$ is called a \emph{Lie lattice} if $\L\otimes_R K=L$, $\L$ is closed under the Lie bracket and $\rho(x)(a)\in \A$ for any $x\in \L$, $a\in \A$. We note that if $\L$ is a Lie lattice, so is $\pi^n\L$ for any $n\geq 0$.\\
\\
We say that $L$ admits a smooth Lie lattice if there exists an affine formal model $\A$ such that $L$ contains a smooth $\A$-Lie lattice $\L$.
\begin{defn}[{\cite[\textsection 6.2]{DhatOne}}]
Let $A$ be an affinoid $K$-algebra with affine formal model $\A$, and let $L$ be a coherent $(K, A)$-Lie algebra with $\A$-Lie lattice $\L$. The Fr\'echet completed enveloping algebra $\w{U_A(L)}$ is defined to be
\begin{equation*}
\w{U_A(L)}=\varprojlim_n \hK{U_{\A}(\pi^n\L)}.
\end{equation*}
\end{defn}
It was shown in \cite[\textsection 6.2]{DhatOne} that this does not depend on the choice of affine formal model and Lie lattice.\\
The key property of $\w{U(L)}$ is that it is a Fr\'echet--Stein algebra in the sense of \cite{ST} whenever $L$ is smooth.
\begin{defn}[{\cite[\textsection 3]{ST}}]
A $K$-Fr\'echet algebra $U$ is called (left, right, two-sided) \emph{Fr\'echet--Stein} if it is isomorphic to an inverse limit $\varprojlim U_n$ for (left, right, two-sided) Noetherian $K$-Banach algebras $U_n$ whose connecting maps are flat (on the right, on the left, on both sides) with dense images.\\
A left $U$-module $M$ is called \emph{coadmissible} if $M\cong \varprojlim M_n$, where $M_n$ is a finitely generated $U_n$-module such that the natural morphism $U_n\otimes_{U_{n+1}}M_{n+1}\to M_n$ is an isomorphism for each $n$.
\end{defn}
For a given Fr\'echet--Stein algebra $U$, we denote the category of coadmissible left $U$-modules by $\C_U$.
\begin{thm}[{\cite[Theorem 3.5]{Bode}}]
Let $A$ be an affinoid $K$-algebra and let $L$ be a smooth $(K, A)$-Lie algebra. Then $\w{U_A(L)}$ is a two-sided Fr\'echet--Stein algebra.
\end{thm}
\subsection{Lie algebroids and Fr\'echet completions}
\begin{defn}
A \emph{Lie algebroid} on a rigid analytic $K$-space $X$ is a pair $(\rho, \sL)$ where
\begin{enumerate}[(i)]
\item $\sL$ is a locally free sheaf of $\O_X$-modules of finite rank on $X_{\rig}$.
\item $\sL$ is a sheaf of $K$-Lie algebras, and
\item $\rho: \sL\to \T$ is an $\O$-linear map of sheaves of Lie algebras such that
\begin{equation*}
[x, ay]=a[x, y]+\rho(x)(a)y
\end{equation*}
for any $a\in \O$, $x, y\in \sL$.
\end{enumerate}
\end{defn}
Given a Lie algebroid $\sL$, there exists a unique sheaf $\w{\sU(\sL)}$ on $X_{\rig}$ such that on each admissible open affinoid subspace $Y\subseteq X$, we have
\begin{equation*}
\w{\sU(\sL)}(Y)= \w{U_{\O(Y)}(\sL(Y))}
\end{equation*}
with the obvious restriction maps to smaller affinoid subdomains (see \cite[Theorem 9.3]{DhatOne} and the remark after \cite[Theorem 4.9]{Bode}).\\
\\
If $X$ is a smooth rigid analytic $K$-space, then its tangent sheaf $\T_X$ (together with $\rho=\id$) is a Lie algebroid, and we write $\w{\D}_X=\w{\sU(\T_X)}$. 
\subsection{Localisation and coadmissible modules}
Given a Lie algebroid $\sL$ on an affinoid $K$-space $X$, write $\w{\sU}=\w{\sU(\sL)}$. We can localise coadmissible modules over the Fr\'echet--Stein algebra $\w{\sU}(X)$ as follows (see \cite[\textsection 8.2]{DhatOne}):\\
For any admissible open affinoid subspace $Y$ of $X$, the functor 
\begin{equation*}
\w{\sU}(Y)\underset{\w{\sU}(X)}{\h{\otimes}}- : \C_{\w{\sU}(X)}\to \w{\sU}(Y)\mathrm{-mod}
\end{equation*} 
sends coadmissible $\w{\sU}(X)$-modules to coadmissible $\w{\sU}(Y)$-modules, where each term is equipped with its canonical Fr\'echet topology (it was shown in \cite[Corollary A.6]{Bitoun} that the completed tensor product $\h{\otimes}$ agrees with the operation $\w{\otimes}$ defined in \cite[\textsection 7.3]{DhatOne}). \\
This gives rise (\cite[Theorem 8.2, Theorem 9.5]{DhatOne}) to a fully faithful exact functor
\begin{equation*}
\Loc: \{\text{coadmissible} \ \w{\sU}(X)\text{--modules}\}\to \{\text{sheaves of} \ \w{\sU}\text{--modules on} \ X_{\rig}\},
\end{equation*}
and we call its essential image $\C_{\w{\sU}}$, the category of coadmissible $\w{\sU}$-modules. 
\begin{defn}[{\cite[Definition 9.4]{DhatOne}}]
If $X$ is an arbitrary rigid analytic $K$-space and $\sL$ is a Lie algebroid on $X$, we say that a sheaf of $\w{\sU(\sL)}$-modules $\M$ on $X_{\rig}$ is coadmissible if there exists an admissible covering $(U_i)$ of $X$ by affinoids such that for each $i$, $\M|_{U_i}\cong \Loc M_i$ for some coadmissible $\w{\sU(\sL)}(U_i)$-module $M_i$.
\end{defn}
If $X$ is smooth, we shorten $\C_{\w{\D}_X}$ to $\C_X$.

\section{Faithfully flat completions of deformable algebras}
\subsection{Statement of the theorem and preliminaries}
\label{fflat}
Recall that a positively filtered $R$-algebra $U$ is called \emph{deformable} if $\gr U$ is flat over $R$. We define its $n$th deformation to be the subring
\begin{equation*}
U_n:=\sum_{i\geq 0} \pi^{in}F_iU.
\end{equation*}
Let $U$ be a deformable $R$-algebra such that $\gr U$ is a commutative Noetherian $R$-algebra and $F_0U$ is $\pi$-adically separated.\\
Note that these assumptions make $F_0U$ a commutative Noetherian $R$-algebra, as it is a quotient of $\gr U$. Moreover, $F_iU$ is a finitely generated $F_0U$-module by \cite[Lemma 6.5]{DhatOne}\\
\\
For example, the properties above are satisfied if $U=U_{\A}(\L)$ is the enveloping algebra of a smooth $(R, \A)$-Lie algebra $\L$ over an affine formal model $\A$ of some affinoid $K$-algebra $A$.\\
It was shown in \cite[Theorem 6.7]{DhatOne} that $\wK{U}=\varprojlim \hnK{U}$ is a two-sided Fr\'echet--Stein algebra.\\
In this section, we prove Theorem \ref{introthmC} from the introduction by proving the corresponding result for deformable algebras.
\begin{thm}
Let $U$ be a deformable $R$-algebra such that $\gr U$ is commutative Noetherian $R$-algebra and $F_0U$ is $\pi$-adically separated. Then the natural morphism $U_K\to \wK{U}$ is faithfully flat.
\end{thm} 
By `faithfully flat' we mean that $\wK{U}$ a faithfully flat $U_K$-module both on the left and on the right. We present the proof for the right module structure, the analogous statement for the left module structure can be shown mutatis mutandis.\\
We note that this result already appears in \cite[Proposition 3.6]{SS} under slightly modified assumptions, but the authors do not follow the last step in the proof and would like to put forward this alternative argument.\\
\\
Note that flatness follows by the same argument as in \cite[Lemma 4.14]{Bode}, so it remains to show faithfulness.\\
\\
Suppose that $N$ is a simple $U_K$-module. We wish to show that $\wK{U}\otimes_{U_K}N\neq 0$, so it suffices to find some $n$ such that $\hnK{U}\otimes_{U_K}N\neq 0$.\\
Let $M$ be a finitely generated $U$-submodule of $N$ such that $N=M_K$, and equip $M$ with a good filtration $F_\bullet M$ (pick a finite generating set $m_1, \dots, m_r$ and set $F_iM=\sum_{j=1}^r F_iUm_j$). We can now form the finitely generated $U_n$-module
\begin{equation*}
M_n=\sum_{i\geq0} \pi^{in}F_iM\subseteq M.
\end{equation*}
This is equipped with the filtration
\begin{equation*}
F_jM_n=\sum_{i=0}^j \pi^{in}F_iM,
\end{equation*}
making it a filtered $U_n$-module.\\
Clearly $M_n\otimes_R K=N$. Moreover, $F_jM_n\subseteq F_jM$ is a finitely generated $F_0U$-module for any $j, n\geq 0$, as $F_0U$ is Noetherian.\\
\\
Recall that \cite[Lemma 3.5]{AW13} provides us with an isomorphism $\xi_n: \gr U\to \gr U_n$ given by multiplication by $\pi^{nj}$ on the $j$th graded piece. The lemma below discusses an analogous morphism for modules.

\begin{lem}
There exists a graded $R$-linear morphism $\mu_n: \gr M\to \gr M_n$ such that
\begin{enumerate}[(i)]
\item $\mu_n(m+F_{j-1}M)=\pi^{nj}m+F_{j-1}M_n$ for all $m\in F_jM$.
\item $\mu_n(u \cdot m)=\xi_n(u)\cdot \mu_n(m)$ for all $u\in \gr U$, $m\in \gr M$.
\item $\mu_n$ is surjective.
\item $\ker (\mu_n|_{\gr_j M})=(\gr_j M)[\pi^{nj}]$.
\end{enumerate}
\end{lem}
\begin{proof}
It is immediate from the definition of the filtrations that (i) gives a well-defined $R$-linear graded morphism $\mu_n$ satisfying (ii).\\
For (iii), note that $\pi^{nj}F_jM+F_{j-1}M_n=F_jM_n$ for any $j\geq 0$, so that $\mu_n(\gr_j M)=\gr_jM_n$.\\
For (iv), let $m\in F_jM$ with the property that $\overline{m}\in \gr_j M$ is annihilated by $\mu_n$. This means that $\pi^{nj}m\in F_{j-1}M_n=\sum_{i=0}^{j-1}\pi^{in}F_iM\subseteq F_{j-1}M$. So $\pi^{nj}\overline{m}=0$, i.e. $\overline{m}\in (\gr_j M)[\pi^{nj}]$. The reverse inclusion is clear.
\end{proof}
\begin{cor}
Suppose that $\pi^n$ annihilates the $\pi$-torsion of $\gr M$. Then $\gr M_n$ is $\pi$-torsionfree.
\end{cor}
\begin{proof}
Let $\overline{x}\in \gr_jM_n$ satisfy $\pi \overline{x}=0$. If $j\geq 1$, then surjectivity of $\mu_n$ (Lemma \ref{fflat}.(iii)) implies that there exists $\overline{y}\in \gr_jM$ such that $\mu_n(\overline{y})=\overline{x}$, and hence 
\begin{equation*}
\mu_n(\pi \overline{y})=\pi \overline{x}=0.
\end{equation*}
Therefore $\overline{y}\in \gr M$ is $\pi$-torsion by Lemma \ref{fflat}.(iv), implying that $\pi^n \overline{y}=0$ by assumption. As $j\geq 1$, it follows that $\overline{x}=\pi^{nj}\overline{y}=0$ as well.\\
\\
As $\gr_0 M_n=F_0 M$ is also $\pi$-torsionfree, we have $(\gr M_n)[\pi]=\oplus_j (\gr_j M_n)[\pi]=0$, as required.
\end{proof}
Since $\gr M$ is finitely generated over the Noetherian ring $\gr U$, it follows from the Corollary that $\gr M_n$ is $\pi$-torsionfree for sufficiently large $n$.
\subsection{Torsionfree deformations and the proof of Theorem \ref{fflat}}
\label{tfdef}
\begin{lem}
Suppose that $\gr M_n$ is $\pi$-torsionfree. Then $N$ is $(1+\pi U_n)$-torsionfree. 
\end{lem}
\begin{proof}
Let $S_n=1+\pi U_n$. Equipping $U$ with its $\pi$-adic filtration, it follows from \cite[Corollary 2.2]{Li} that $S_n$ is an Ore set in $U$ and hence in $U_K$. In particular, the set of $S_n$-torsion elements in $N$ is a $U_K$-submodule of $N$, and by simplicity is either $0$ or $N$ itself. Assume therefore that $N$ and hence $M_n$ is $S_n$-torsion.\\
\\
We now claim that $F_jM_n=\pi F_jM_n$ for any $j\geq 0$. Once we have proved the claim, finite generation of $F_jM_n$ over the Noetherian $\pi$-adically complete ring $F_0U$ forces $F_jM_n=0$ by \cite[Corollary 10.19]{AM}, so $M_n=0$, which provides us with the desired contradiction, as $M_n\otimes_R K=N\neq 0$.\\
\\
Let $m\in F_jM_n$. As $M_n$ is $S_n$-torsion, there exists $u\in U_n$ such that $(1-\pi u)m=0$, so $m=\pi um$.\\
Since $\gr M_n$ is $\pi$-torsionfree, the same is true for $F_iM_n/F_{i-1}M_n$ for each $i$, and thus $M_n/F_jM_n$ is also $\pi$-torsionfree. In particular, $\pi um\in F_jM_n$ implies $um\in F_jM_n$ and thus $m\in \pi F_jM_n$, as required.
\end{proof}
\begin{proof}[{Proof of Theorem \ref{fflat}}]
By Corollary \ref{fflat} and Lemma \ref{tfdef}, there exists $t$ such that $N$ is $S_n$-torsionfree for any $n\geq t$. Then $\h{M_n}\neq 0$ by \cite[Theorem 10.17]{AM}. As $M_n$ is $\pi$-torsionfree, so is $\h{M_n}\cong\h{U_n}\otimes_{U_n}M_n$, and hence $\h{M_n}\otimes_R K\cong \hnK{U}\otimes_{U_K}N\neq 0$ for any $n\geq t$. Therefore $\wK{U}\otimes_{U_K}N\neq 0$, as required.
\end{proof}
\begin{cor}
Let $A$ be an affinoid $K$-algebra and let $L$ be a smooth $(K, A)$-Lie algebra. Then $U_A(L)\to \w{U_A(L)}$ is faithfully flat.
\end{cor}
\begin{proof}
Flatness was already proven in \cite[Lemma 4.14]{Bode}, so it suffices to show faithfulness.\\
If $L$ admits a smooth Lie lattice, this follows directly from Theorem \ref{fflat}. In the general case, note that there exists a finite affinoid covering $(\Sp A_i)$ of $\Sp A$ such that $A_i\otimes_A L$ is free (and in particular admits a smooth Lie lattice) for each $i$. We thus obtain the commutative diagram
\begin{equation*}
\begin{xy}
\xymatrix{
U_A(L)\ar[r] \ar[d] &\w{U_A(L)}\ar[d]\\
\oplus U_{A_i}(A_i\otimes_A L)\ar[r] & \oplus \w{U(A_i\otimes L)}
}
\end{xy}
\end{equation*}
where the left vertical arrow is faithfully flat due to the isomorphism $U(A_i\otimes_A L)\cong A_i\otimes_A U_A(L)$ (see \cite[Proposition 2.3]{DhatOne}). Theorem \ref{fflat} thus implies that $\oplus \w{U(A_i\otimes_A L)}$ is faithfully flat over $U_A(L)$: if $N$ is a non-zero $U(L)$-module, then there exists some $i$ such that $U(A_i\otimes L)\otimes_{U(L)} N\neq 0$, and hence $\w{U(A_i\otimes L)}\otimes_{U(L)} N\neq 0$. Therefore, if $N$ is a $U(L)$-module such that $\w{U(L)}\otimes N=0$, it follows from $\oplus \w{U(A_i\otimes L)}\otimes_{U(L)}N=0$ that $N=0$.
\end{proof}
If $A$ is a smooth affinoid $K$-algebra and $L=\T(\Sp A)$, this now proves Theorem \ref{introthmC} from the introduction.
\section{Auslander--Gorenstein rings and completed enveloping algebras}

\subsection{Faithfully flat descent}\label{AGfflat}
\begin{defn}Let $A$ be a ring.
\begin{enumerate}[(i)]
\item The \emph{grade} of an $A$-module $M$ is 
\begin{equation*}
j(M)=\min\{i: \Ext^i_A(M, A)\neq 0\},
\end{equation*}
and $\infty$ if no such $i$ exists.
\item We say that $A$ satisfies the \emph{Auslander condition} if for every Noetherian $A$-module $M$ and any $i\geq 0$, we have $j(N)\geq i$ whenever $N$ is a (right) submodule of $\Ext^i_A(M, A)$.
\item A two-sided Noetherian ring is called \emph{Gorenstein} if it has finite left and right injective dimension.
\item A two-sided Noetherian ring is called \emph{Auslander--Gorenstein} if it satisfies the Auslander condition and has finite left and right injective dimension.
\end{enumerate}
\end{defn}

The proof of the following straightforward lemma can be pieced together from the literature (see for example \cite[Theorem 3.3]{LS06} and \cite[Theorem 1.2]{BBP}) but as far as we know it has never been written down in this generality in a single place.

\begin{lem} If $S\to S'$ is a faithfully flat homomorphism of rings with $S'$ Auslander--Gorenstein then $S$ is also Auslander--Gorenstein. Moreover the dimension of $S$ is bounded above by the dimension of $S'$.
\end{lem}

\begin{proof} First we show that $S$ is left Noetherian. Suppose that $(I_n)$ is an ascending chain of left ideals in $S$. For each $n\in \mathbb{N}$ let $J_n$ be the image of $S'\otimes_S I_n\to S'$ so that $(J_n)$ is an ascending chain of left ideals in $S'$. Since $S'$ is left Noetherian the chain $J_n$ must terminate. Because $S\to S'$ is faithfully flat $J_n/J_{n-1}\cong S'\otimes_S I_n/I_{n-1}\neq 0$ whenever $I_n/I_{n-1}\neq 0$. Thus the chain $I_n$ must terminate. \\
\\
By symmetry (or by considering the opposite rings) $S$ is also right Noetherian. \\
\\
Next we show that $S$ satisfies the Auslander condition. Let $N$ be a right $S$-submodule of $\Ext^i_S(M, S)$ for some finitely generated left $S$-module $M$, and let $j<i$. Since $S\to S'$ is flat there is an isomorphism $S'\otimes_S \Ext_S^j(N,S)\cong \Ext_{S'}^j(N\otimes_S S',S')$. The latter is zero since $S'$ satifies the Auslander condition and $N\otimes_S S'$ is isomorphic to a submodule of $\Ext_S^i(M,S)\otimes_S S'\cong \Ext_{S'}^i(S'\otimes_S M,S')$ by the flatness of $S\to S'$ again. Since $S\to S'$ is faithfully flat we may deduce that $\Ext_S^j(N,S)=0$ as required.  \\
\\
Finally suppose that $d=\mathrm{injdim}_{S'} S'$. For each cyclic $S$-module $M$ we can compute $\Ext^{d+1}(M,S)\otimes_S S'\cong \Ext^{d+1}(S'\otimes_S M,S')=0$. Since $S\to S'$ is faithfully flat we can deduce $\Ext^{d+1}(M,S)=0$ and so, using \cite[p55]{Jans}, that $\mathrm{injdim}_S S\leq d$.
\end{proof}

\subsection{Smooth affinoids and Gorenstein formal models}\label{GorensteinFormalModel}

Recall that Raynaud's Theorem \cite[Theorem 4.1]{BL1} establishes an equivalence of categories $\X\mapsto \X_{rig}$ between the category of quasi-compact admissible formal $R$-schemes localised by admissible formal blowing-ups, and the category of quasi-compact rigid analytic spaces over $K$. (Recall that all our rigid analytic spaces are assumed to be quasi-separated.)

\begin{prop} Suppose that $X$ is a smooth quasi-compact rigid analytic space over $K$. Given any quasi-compact formal model $\X$ of $X$ over $R$ there is an admissible formal blowing-up $\X^\dag\to \X$ of formal $R$-schemes with $\X^\dag$ Gorenstein. 
\end{prop}

\begin{proof} By \cite[Theorem 1.4]{Hartl} there is a finite separable field extension $K'$ of $K$ with ring of integers $R'$, a quasi-compact strictly semi-stable formal $R'$-scheme $\X'$, and a composition of morphisms of quasi-compact admissible formal $R$-schemes $\X'\to \X^\dag\to \X$ such that $\X^\dag \to \X$ is an admissible formal blowing-up and $\X'\to \X^\dag$ is flat and surjective. Since being Gorenstein is a local condition and $\X'\to \X^\dag$ is faithfully flat it suffices by Lemma \ref{AGfflat} to show that $\X'$ is Gorenstein. But $\X'$ is a regular scheme (\cite[Remark 1.1.1]{Hartl}) and so Gorenstein. 
\end{proof}

\subsection{Completed enveloping algebras over Gorenstein algebras}\label{CompleteEnvAlg}

\begin{lem} Suppose that $A$ is a commutative Gorenstein $k$-algebra for some commutative ring $k$, and $L$ is a smooth $(k,A)$-Lie algebra of rank $r$. Then $U(L)$ is Auslander--Gorenstein of dimension at most $\dim A+r$. 
\end{lem}

\begin{proof} By \cite[Theorem 3.1]{Rinehart} there is a positive filtration on $U(L)$ such that $\gr U(L)\cong \Sym(L)$ is commutative. By \cite[Th\'eor\`eme 4.4, Remarque 4.5]{Lev85} it thus suffices to show that $\Sym(L)$ is Auslander--Gorenstein of dimension at most $\dim A+r$. \\
\\
Since $L$ is a finitely generated projective $(k,A)$-module there is a cover of $\Spec(A)$ by basic open subsets $D(f_1),\ldots,D(f_m)$ such that $A_{f_i}\otimes_A L$ is a free $A_{f_i}$-module of rank $r$ for each $i=1,\ldots,m$. Since $\Sym(L)\to \oplus_{i=1}^m\Sym(A_{f_i}\otimes_A L)$ is faithfully flat we can use Lemma \ref{AGfflat} to reduce to the case that $L$ is free over $A$, i.e. $\Sym(L)$ is isomorphic to a polynomial ring $A[t_1,\ldots,t_r]$. Thus we are done by \cite[Corollary 1]{WIT}. 
\end{proof}

\begin{thm} Suppose that $A$ is a smooth $K$-affinoid algebra over $K$ with affine formal model $\A$ and that $\L$ is a smooth $(R,\A)$-Lie algebra of rank $r$. There is an integer $m\geq 0$ such that $\hK{U(\pi^n\L)}$ is Auslander--Gorenstein of dimension at most $\dim A+r$ for each $n\geq m$. \end{thm}

\begin{proof} First we establish the result when $\A$ is Gorenstein with $m=0$. In this case it follows from the Lemma that $U(\L)$ is Auslander--Gorenstein of dimension at most $\dim \A+r=\dim A+1+r$. Thus $U(\L)/\pi U(\L)$ is Auslander--Gorenstein of dimension at most $\dim A+r$ by \cite[Proposition 1.3 and Proposition 2.1]{ASZ2}.\\
\\
Now $\hK{U(\pi^n\L)}$ is a complete doubly filtered $K$-algebra and \[\Gr(\hK{U(\pi^n\L)})\cong U(\L)/\pi U(\L)\] by \cite[Lemma 3.7]{AW13}. Thus $\hK{U(\pi^n\L)}$ is Auslander--Gorenstein of dimension at most $\dim A+r$ by \cite[Theorem 3.9]{Bj89}.\\
\\
In the general case Proposition \ref{GorensteinFormalModel} shows that there is an admissible formal blowing-up $\X\to \Spf(\A)$ over $R$ with $\X$ Gorenstein. Let $\{\Spf(\A_i)\}$ be an affine cover of $\X$. By \cite[Proposition 7.6]{DhatOne} and the proof of \cite[Lemma 7.6(b)]{DhatOne} there is a positive integer $m$ such that each $\A_i$ is $\L$-stable and each $(\Spf \A_i)_{rig}$ is a $\pi^n\L$-accessible subdomain of $X$ for each $n\geq m$. Thus each natural map \[ \hK{U(\pi^n\L)}\to \bigoplus \hK{U(\A_i\otimes_{\A}\pi^n\L)}\] with $n\geq m$ is faithfully flat by \cite[Theorem 4.9(b)]{DhatOne}. The result now follows from Lemma \ref{AGfflat} and the case $\A$ is Gorenstein.
\end{proof}
This proves Theorem \ref{introthmA}.(i).

\section{Dimension theory for coadmissible modules}

\subsection{Review}\label{ReviewDim}
We slightly generalise the exposition of \cite[\S 8]{ST} which introduced a dimension theory for Fr\'echet--Stein algebras with the property that each member of the defining family of Banach algebras is Auslander regular of global dimension bounded by a universal constant; we relax this last condition to Auslander--Gorenstein with self-injective dimension bounded by a universal constant.\\
\\
We suppose throughout this section that $U=\invlim U_n$ is a two-sided Fr\'echet--Stein algebra.

\begin{lem}[{\cite[Lemma 8.4]{ST}}] For any coadmissible left $U$-module $M$ and any integer $l\ge 0$ the $U$-module $\Ext^l_U(M,U)$ is coadmissible with \[ \Ext^l_U(M,U)\otimes_U U_n \cong \Ext^l_{U_n}(U_n\otimes_UM,U_n) \] for any $n\in \mathbb{N}$.
\end{lem}

\begin{defn} We say that $U$ is \emph{coadmissibly Auslander--Gorenstein} (or \emph{c-Auslander--Gorenstein}) of dimension at most $d$ if $d$ is a non-negative integer such that each $U_n$ is Auslander--Gorenstein with self-injective dimension at most $d$. 
\end{defn}

It follows easily from the Lemma that if $A$ is c-Auslander--Gorenstein of dimension at most $d$ then every coadmissible $A$-module $M$ satisfies Auslander's condition; that is for every integer $l\ge 0$ every coadmissible submodule $N$ of $\Ext^l_A(M,A)$ has grade at least $l$. It is also an easy consequence that every (non-zero) coadmissible $A$-module has grade at most $d$. \\
\\
If $A$ is a smooth $K$-affinoid algebra and $L$ is a $(K,A)$-Lie algebra that admits a smooth lattice of rank $r$ then $\w{U(L)}$ is c-Auslander--Gorenstein of dimension at most $\dim A+r$ by Theorem \ref{CompleteEnvAlg}. In particular if $\Der_K(A)$ admits a smooth lattice then $\w{\D}(\Sp A)$ is c-Auslander--Gorenstein of dimension at most $2\dim A$. 

\begin{defn} Suppose that $A$ is a smooth $K$-affinoid algebra and $L$ is a $(K,A)$-Lie algebra that admits a smooth lattice of rank $r$. Writing $U$ for $\w{U(L)}$, the \emph{dimension} of a (non-zero) coadmissible $U$-module $M$ is defined by \[ d_U(M):= \dim A+r - j_{U}(M). \] We will sometimes suppress the subscript $U$ and simply write $d(M)$ if this will not cause confusion.
\end{defn}



\subsection{Left-right comparison}
\label{leftright}
Let $A$ be an affinoid $K$-algebra and let $L$ be a $(K, A)$-Lie algebra that admits a smooth Lie lattice of rank $r$. Recall \cite[Theorem 3.4]{DhatTwo} that there is an equivalence of categories between coadmissible left $\w{U(L)}$-modules and coadmissible right $\w{U(L)}$-modules, given by $\Omega_L\otimes_A-$ and $\Hom_A(\Omega_L, -)$, where 
\begin{equation*}
\Omega_L=\Hom_A\left(\bigwedge\nolimits^r L, A\right).
\end{equation*}  

\begin{lem} For each coadmissible left $\w{U(L)}$-module $M$ there is a natural isomorphism \[\Ext^j_{\w{U(L)}}(\Omega_L\otimes_A M, \w{U(L)})\cong \Hom_A({\Omega_L}, \Ext^j_{\w{U(L)}}(M,\w{U(L)})\ \forall j\geq 0. \] In particular $d(M)=d(\Omega_L\otimes_A M)$. \end{lem}

\begin{proof}
As $\Omega_L$ is a projective $A$-module, the left hand side is the $j$th cohomology of $\mathrm{R}\Hom_{\w{U(L)}}(\Omega_L\otimes^{\mathbb{L}}_AM, \w{U(L)})$, while the right hand side is the $j$th cohomology of $\mathrm{R}\Hom_A(\Omega_L, \mathrm{R}\Hom_{\w{U(L)}}(M, \w{U(L)})$. The natural isomorphism thus follows directly from the derived tensor-Hom adjunction \cite[Theorem 18.6.4.(vii)]{KS}.\\
\\
Because $\Omega_L$ is an invertible $A$-module, 
\begin{equation*}
\Ext^j_{\w{U(L)}}(M,\w{U(L)})=0 \ \text{if and only if}\ \Hom_A(\Omega_L, \Ext^j_{\w{U(L)}}(M, \w{U(L)}))=0
\end{equation*}
for each $j\geq 0$, and hence $d(M)=d(\Omega_L\otimes_A M)$.
\end{proof}
\subsection{Dimension theory for \ts{\w{\sU(\sL)}}}
\label{dimtheorysheaf}
Let $\sL$ be a Lie algebroid on a smooth, equidimensional rigid $K$-analytic space $X$. Let $X_w(\sL)$ denote the set of affinoid subspaces $Y$ of $X$ such that $\sL(Y)$ admits a smooth Lie lattice. By \cite[Lemma 9.3]{DhatOne} $X_w(\sL)$ is a basis for the topology on $X$ i.e. every admissible open in $X$ has an admissible cover by objects in $X_w(\sL)$.

\begin{prop} For each $t\geq 0$ there is a functor $\sExt^t_{\w{\sU(\sL)}}(-, \w{\sU(\sL)})$ from coadmissible left $\w{\sU(\sL)}$-modules on $X$ to coadmissible right $\w{\sU(\sL)}$-modules on $X$ such that \[ \sExt^t_{\w{\sU(\sL)}}(\M, \w{\sU(\sL)})(Y)= \Ext^t_{\w{U(\sL(U))}}(\M(Y),\w{U(\sL(Y))}) \] for each coadmissible left $\w{\sU(\sL)}$-module $\M$ and each $Y\in X_w(\sL)$.
\end{prop}

\begin{proof}
Let $\M$ be a coadmissible $\w{\sU(\sL)}$-module and suppose that $Z\subset Y$ are in $X_w(\sL)$. Let $\A$ be an affine formal model in $\O(Y)$ such that $\sL(Y)$ admits a smooth Lie lattice $\L$. By replacing $\L$ by $\pi^m\L$ for some positive integer $m$, we may assume that $Z$ is $\pi^n\L$-accessible for all $n\geq 0$ (see \cite[Proposition 7.6]{DhatOne}).\\
Let $\B$ be an affine formal model in $\O(Z)$ such that $\B\otimes_{\A}\L$ is a $\B$-Lie lattice in $\sL(Z)=B\otimes_A\sL(Y)$.\\
\\
For each $n\geq 0$, let $U_n$ and $V_n$ denote the $K$-Banach algebras $\hK{U(\pi^n\L)}$ and $\hK{U(\B\otimes \pi^n\L)}$ respectively and let $U:=\invlim U_n=\w{U(\sL(Y))}$ and $V:=\invlim V_n=\w{U(\sL(Z))}$. Now, using \cite[Theorem 9.4]{DhatOne}, we see that $M_n(Y):=U_n\otimes_{U}\M(Y)$ is a finitely generated left $U_n$-module, $M_n(Z):=V_n\otimes_{V}\M(Z)$ is a finitely generated left $V_n$-module and $M_n(Z)\cong V_n\otimes_{U_n}M_n(Y)$. Morever by \cite[Theorem 4.8]{DhatOne} $U_n\to V_n$ is flat on both sides for all positive integers $n$. Thus $\Ext^t_{V_n}(M_n(Z),V_n)\cong \Ext^t_{U_n}(M_n(Y),U_n)\otimes_{U_n}V_n$ for each $t\geq 0$. \\
\\
By Lemma \ref{ReviewDim}, $\Ext^t_{U}(\M(Y),U)$ is a coadmissible right $U$-module such that 
\begin{equation*}
\Ext^t_{U}(\M(Y),U)\otimes_{U}U_n\cong \Ext^t_{U_n}(M_n(Y),U_n)
\end{equation*} 
and $\Ext^t_{V}(\M(Z),V)$ is a coadmissible $V$-module with $\Ext^t_{V}(\M(Z),V)\otimes_{V}V_n\cong \Ext^t_{V_n}(M_n(Z),V_n)$. Thus we can compute \begin{eqnarray*}\Ext^t_V(\M(Z),V)\otimes V_n & \cong & \Ext^t_{V_n}(M_n(Z),V_n) \\
                                                              & \cong & \Ext^t_{U_n}(M_n(Y),U_n)\otimes_{U_n}V_n\\
                                                              & \cong & \Ext^t_U(\M(Y),U)\otimes_U V_n \end{eqnarray*}
Now we see that $\Ext^t_V(\M(Z),V)\cong \Ext^t_U(\M(Y),U)\w\otimes_U V$ by (the proof of) \cite[Lemma 7.3]{DhatOne}. It follows, using \cite[Theorem 8.2, Theorem 8.4]{DhatOne} that the presheaf on $Y_w$ that sends $Z\in Y_w$ to $\Ext^t_{\w{U(\sL(Z))}}(\M(Z),\w{U(\sL(Z))})$ is a coadmissible sheaf of $\w{\sU(\sL)}$-moudules on $Y_w$. Applying \cite[Theorem 9.1, Theorem 9.4]{DhatOne} completes the proof of the Proposition.  
\end{proof}
The analogous statement for coadmissible right modules can be proven using the same argument.
\begin{defn} Suppose that $\mathcal{U}$ is an admissible cover of $X$ by afffinoid subspaces in $X_w(\sL)$. For each (non-zero) coadmissible $\w{\sU(\sL)}$-module $\M$ we define the \emph{dimension} of $\M$ with respect to $\mathcal{U}$ by \[ d_{\mathcal{U}}(\M)= \sup \{d(\M(Y))\st Y\in \mathcal{U}\}.\]\end{defn}

\begin{lem} Suppose that $\mathcal{U}$ and $\mathcal{V}$ are two admissible covers of $X$ by affinoid subspaces in $X_w(\sL)$. Then $d_\mathcal{U}(\M)= d_{\mathcal{V}}(\M)$. 
\end{lem}

\begin{proof} Since $X$ is quasi-separated, for any pair $Y\in \mathcal{U}$ and $Z\in \mathcal{V}$ that does not intersect trivially we can cover $Y\cap Z$ by a finite set of affinoid subspaces in $X_{w}(\L)$. Thus we may reduce to the case that $\V$ is a refinement of $\U$ and every element of $\U$ has an admissible cover by elements of $\V$. \\
\\
Now suppose $Y\in \mathcal{U}$ is covered by $Z_1,\ldots,Z_k\in \mathcal{V}$. Then by the Proposition $\sExt_{\w{\sU(\sL|_Y)}}^j(\M|_Y, \w{\sU(\sL|_Y)})$ is a coadmissible $\w{\sU(\sL|_Y)}$-module for each $j\geq 0$. Since $\w{U(\sL(Y))} \to \oplus_{i=1}^k \w{U(\sL(Z_i))}$ is c-faithfully flat by \cite[Theorem 7.7(b)]{DhatOne} it follows that $j(\M(Y))=\inf \{j(\M(Z_i))| \M(Z_i)\neq 0\}$.   \\
As $\dim Y=\dim Z_i$ for each $i$ by equidimensionality, the result follows.
\end{proof}

It follows that we may define the dimension of $\M$ by $d(\M)=d_{\mathcal{U}}(\M)$ for any choice $\U$ of admissible cover of $X$ by affinoid subspaces in $X_w(\sL)$.

\section{Bernstein's inequality}

\subsection{Dimension and pushforward along a closed embedding}\label{Kashiwaradim}
Let $\iota: Y\to X$ be a closed embedding of smooth rigid analytic $K$-spaces. In \cite{DhatTwo}, the first and the third author produced a functor
\begin{equation*}
\iota_+: \C_Y\to \C_X.
\end{equation*}
The construction of $\iota_+$ rests on the case when $\iota: Y=\Sp A/I\to \Sp A=X$ and $L=\T(X)$ admits an $I$-\emph{standard basis}, i.e. there exists an $A$-basis $\{x_1, \dots, x_d\}$ for $L$ and a generating set $\{f_1, \dots, f_r\}$ for $I$ with $r\leq d$ such that $x_i\cdot f_j=\delta_{ij}$ for any $1\leq i\leq d$, $1\leq j\leq r$. \\
Any closed embedding of smooth rigid analytic $K$-varieties is locally of this form by \cite[Theorem 6.2]{DhatTwo}, and we have
\begin{equation*}
\iota_+\M(X)=\M(Y)\underset{\w{\D}(Y)}{\w{\otimes}} \w{U(L)}/I\w{U(L)}
\end{equation*}
for any coadmissible right $\w{\D}_Y$-module $\M$ -- the corresponding functor for left modules is obtained via side-changing operations.\\
We refer to \cite{DhatTwo} for details.\\
\\
We will now show that $\iota_+$ respects our dimension function in a natural way, allowing us to reduce many statements about modules on smooth affinoid spaces to the corresponding statements on polydiscs. 

\begin{lem}
Suppose that $A$ is a smooth affinoid $K$-algebra with affine formal model $\A$, and $\L$ is a smooth $(R, \A)$-Lie algebra. Let $f\in \A$ such that $\L\cdot f=\A$, and write $\C=C_{\L}(f)=\{x\in \L: x\cdot f=0\}$. Let $I=(fA)\cap \A$.
\begin{enumerate}[(i)]
\item $B:=A/fA$ is a smooth affinoid algebra with affine formal model $\B:=\A/I$, and $\C/I$ is a smooth $(R, \B)$-Lie algebra.
\item If $\Sp A$ is connected, then $f$ is a regular central element in $\hK{U_{\A}(\C)}$.
\end{enumerate}
\end{lem}
\begin{proof}
By assumption, there exists $x\in \L$ such that $x\cdot f=1$. Applying the anchor map, we also have $\partial\in \mathrm{Der}_K(A)$ such that $\partial\cdot f=1$, and by \cite[Lemma 4.1]{DhatTwo}
\begin{equation*}
\mathrm{Der}_K(A)=A\partial\oplus C,
\end{equation*}
where $C=\{\xi\in \mathrm{Der}(A): \xi\cdot f=0\}$. In particular, it follows from smoothness that $C$ is a projective $A$-module. \\
By the second fundamental exact sequence (\cite[Proposition 1.2]{BLR3}), we have 
\begin{equation*}
\mathrm{Der}_K(B)=C/(f),
\end{equation*}
which is projective, so $B$ is smooth. Moreover, $\B$ is easily seen to be an affine formal model.\\
\\
In the same vein, applying \cite[Lemma 4.1]{DhatTwo} to $\L$ yields
\begin{equation*}
\L=\A x\oplus \C,
\end{equation*}
so $\C$ is projective over $\A$, and the $(R, \B)$-Lie algebra $\C/I$ is also a projective $\B$-module.\\
\\
For (ii), note that $f$ is central in $\hK{U(\C)}$ by definition of $\C$. So it suffices to show that $f\cdot P\neq 0$ for any non-zero $P\in \hK{U(\C)}$. In particular, this only depends on the $A$-module structure of $\hK{U(\C)}$. But it follows from \cite[Theorem 3.1]{Rinehart} that 
\begin{equation*}
\hK{U(\C)}\cong \hK{\Sym_{\A}(\C)}
\end{equation*}
as $A$-modules. Since $\C$ is a projective finitely generated $\A$-module, it is a direct summand of a free $\A$-module of finite rank, $\L'$. By functoriality, $\hK{\Sym \C}$ embeds into $\hK{\Sym \L'}$, so we are done if we can show that $f$ is regular in $A$.\\
\\
To show that $f$ is regular in $A$, consider the annihilators
\begin{equation*}
\Ann_A(f)\subseteq \Ann_A(f^2)\subseteq \dots
\end{equation*}
By Noetherianity of $A$, this chain of ideals stabilizes, so that 
\begin{equation*}
\Ann(f^r)=\Ann(f^{r+1})=\dots
\end{equation*} 
for some integer $r$.\\
As $X=\Sp A$ is smooth, we can consider the Fr\'echet--Stein algebra $D=\w{\D}(X)$. We claim that $\Ann_A(f^r)$ is a $D$-submodule of $A$, from which $\Ann(f)\subseteq \Ann(f^r)=\{0\}$ follows because of \cite[Proposition 7.4]{DhatTwo} (note that $\Ann(f^r)\neq A$ as $f^r\neq 0$ by reducedness of $\Sp A$).\\
If $a\in \Ann(f^r)$ and $\xi\in \mathrm{Der}(A)$, then
\begin{equation*}
\xi(a)\cdot f^{r+1}=\xi(af^{r+1})-af^r\xi(f)=0,
\end{equation*}
so that $\xi(a)\in \Ann(f^{r+1})=\Ann(f^r)$. By \cite[Theorem 7.3]{DhatTwo}, it follows that $\Ann(f^r)$ is a $D$-submodule of $A$, as required.
\end{proof}
\begin{prop} 
Suppose that $A$ is a smooth, connected affinoid algebra with affine formal model $\A$ and $\L$ is a smooth $(R,\A)$-Lie algebra. Let $F=\{f_1,\ldots,f_r\}$ be a subset of $A$ such that $\L\cdot(f_1,\ldots,f_r)=\A^r$ and write $\C=C_\L(F)=\{x\in \L: x\cdot f=0 \ \forall f\in F\}$. If $M$ is a finitely generated $\hK{U(\C)}/(F)$-module then \[  j_{\hK{U(\L)}}(\hK{U(\L)}\otimes_{\hK{U(\C)}}M)=j_{\hK{U(\C)}/(F)}(M)+r.\]
\end{prop}

\begin{proof}
We first introduce some notation. For $i=0, \dots, r$, write $A_i=A/(f_1, \dots, f_i)$, $I_i=(\sum_{j=1}^i f_jA)\cap \A$, and $\A_i=\A/I_i$. Let $\C_i=C_{\L}(\{f_1, \dots, f_i\})$ and let $\L_i=\C_i/I_i$. By applying the preceeding lemma repeatedly, each $A_i$ is smooth, with affine formal model $\A_i$, and $\L_i$ is a smooth $(R, \A_i)$-Lie algebra. Note also that $\L_0=\L$ and $\L_r=\C/(F)$.\\
In fact, \cite[Lemma 4.1]{DhatTwo} gives an explicit description of $\L_i$: if $x_1, \dots, x_r\in \L$ such that $x_if_j=\delta_{ij}$, then
\begin{equation*}
\L_i=\oplus_{j=i+1}^r \A_ix_j\oplus \C/I_i.
\end{equation*}
In particular, note that the natural map 
\begin{equation*}
\L_{i}\to C_{\L_{i-1}}(f_{i})/I_i
\end{equation*}
is an isomorphism for $i=1, \dots, r$, and $\hK{U(\C)}/(F)\cong \hK{U(\L_r)}\cong \hK{U(C_{\L_{r-1}}(f_r))}/(f_r)$.\\
\\
It thus suffices to prove the following claim:\\
If $M$ is finitely generated $\hK{U(\L_i)}$-module then
\begin{equation*}
j_{\hK{U(\L_{i-1})}}\left(\hK{U(\L_{i-1})}\underset{\hK{U(C_{\L_{i-1}}(f_i))}}{\otimes}M\right)=j_{\hK{U(\L_i)}}(M)+1.
\end{equation*}
Write $U=\hK{U(C_{\L_{i-1}}(f_i))}$, $V=\hK{U(\L_{i-1})}$. Since by \cite[Corollary 4.3]{DhatTwo} $U\to V$ is faithfully flat, and $\Ext^j_{U}(M,U)\otimes_UV\cong \Ext^j_V(V\otimes_U M, V)$ we see that  $j_U(M)=j_V(V\otimes_UM)$. Thus it suffices to show that $j_U(M)=j_{U/(f_i)}(M)+1$. By the preceeding lemma, $f_i$ is a regular central element in $U$, so we can apply \cite[Lemma 1.1]{ASZ2} to conclude the proof.
\end{proof}

\begin{cor} Suppose that $I$ is an ideal in the smooth, connected affinoid $K$-algebra $A$ and let $L$ be a $(K,A)$-Lie algebra which admits an $I$-standard basis. Write $L_Y$ for the $(K,A/I)$-Lie algebra $N_L(I)/IL$. If $M$ is a (non-zero) coadmissible $\w{U(L_Y)}$-module then $d_{\w{U(L)}}(\iota_+M)= d_{\w{U(L_Y)}}(M)+ \dim A - \dim A/I$.
\end{cor}

\begin{proof}
By Lemma \ref{leftright}, we can consider the case where $M$ is a right coadmissible module. Let $\{x_1, \dots, x_d\}$ be an $I$-standard basis and let $F=\{f_1, \dots, f_r\}\subset A$ be the corresponding generating set. Rescaling the $x_i$ and $f_j$ if necessary, we can assume that there exists an affine formal model $\A$ of $A$ such that $\L:=\sum \A x_i$ is a free Lie lattice in $L$.\\ 
Write $U_n=\hK{U(\pi^n\L)}$. Recall that \cite[Lemma 5.8]{DhatTwo} provides an isomorphism of Fr\'echet--Stein algebras
\begin{equation*}
\w{U(L_Y)}\cong \w{U(C_L(F))}/(F).
\end{equation*}
Denote $\hK{U(C_{\pi^n\L}(F))}$ by $V_n$, so $\w{U(L_Y)}\cong \varprojlim V_n/(F)$ exhibits $\w{U(L_Y)}$ as a Fr\'echet--Stein algebra.\\
\\
Now
\begin{equation*}
j_{\w{U(L)}}(\iota_+M)=j_{U_n}\left(\iota_+M\underset{\w{U(L)}}{\otimes} U_n\right)
\end{equation*}
for sufficiently large $n$ by Lemma \ref{ReviewDim}, and by definition of $\iota_+$,
\begin{equation*}
\iota_+M\underset{\w{U(L)}}{\bigotimes} U_n\cong M\underset{\w{U(L_Y)}}{\bigotimes} \frac{V_n}{(F)}\underset{V_n/(F)}{\bigotimes}\frac{U_n}{IU_n}.
\end{equation*}
Therefore Proposition \ref{Kashiwaradim} implies that
\begin{equation*}
j_{U_n}(\iota_+M\otimes U_n)=j_{V_n/(F)}\left(M\otimes_{\w{U(L_Y)}} V_n/(F)\right)+r.
\end{equation*}
Since $\dim A-\dim A/I=r$, it follows that for sufficiently large $n$,
\begin{align*}
d_{\w{U(L)}}(\iota_+M)&=2\dim A-j_{U_n}\left(\iota_+M\otimes U_n\right)\\
& =2\dim A-j_{V_n/(F)}(M\otimes V_n/(F))-r\\
& = 2r +(2\dim A/I-j_{\w{U(L_Y)}}(M))-r\\
&=d_{\w{U(L_Y)}}(M)+r
\end{align*}
as required.
\end{proof}

\begin{thm} Suppose that $\iota\colon Y\to X$ is a closed embedding of smooth, equidimensional $K$-affinoid spaces. Then for every non-zero coadmissible $\w{\D}_Y$-module $\M$ the dimension of $\iota_+\M$ is given by \[ d(\iota_+\M) = d(\M) + \dim X - \dim Y. \]
\end{thm}
\begin{proof}
By \cite[Theorem 6.2]{DhatTwo}, there exists an affinoid covering $(X_i)$ of $X$, with $X_i=\Sp A_i$ connected such that the $(K, A_i)$-Lie algebra $L_i=\T(X_i)$ admits an $I_i$-standard basis, where $I_i\subset A_i$ is the vanishing ideal of $Y\cap X_i$ -- the conditions in the reference are satisfied by smoothness of $Y$. Moreover, $\T_Y(Y\cap X_i)\cong N_{L_i}(I_i)/I_iL_i$ by dualizing \cite[Proposition 1.2]{BLR3}. Hence the claim follows from Corollary \ref{Kashiwaradim}. 
\end{proof}

\subsection{Proof of Bernstein's inequality}\label{Bernsteinsineq}

\begin{prop} Let $X=\Sp K\langle x_1,\ldots, x_d\rangle$ be a polydisc. Each non-zero coadmissible $\w{\D}_X$-module has dimension at least $d$.
\end{prop}

\begin{proof} Let $\L$ be the $\R\langle x_1,\ldots, x_d\rangle$-submodule of $\Der_K(K\langle x_1,\ldots,x_d\rangle)$ spanned by $\{\partial_1,\ldots,\partial_d\}$. Then $\L$ is a smooth Lie lattice in $\T(X)$ and write $D_n=\hK{U(\pi^n\L)}$ and $D=\w{\D}(X)$. Then $D=\invlim D_n$ is a presentation of $D$ as a Fr\'echet--Stein algebra. \\
\\
Let $\M$ be a coadmissible $\w{\D}_X$-module and write $M:=\M(X)$, a coadmissible $D$-module, and $M_n=D_n\otimes_D M$. Since $X\in X_w(\T(X))$ it suffices to show that $j_D(M)\leq d$. Since $\Ext^j_{D}(M,D)\otimes_{D}D_n \cong \Ext^j_{D_n}(M_n,D_n)$ by Lemma \ref{ReviewDim} it suffices to show that $j_{D_n}(M_n)\leq d$ whenever $M_n\neq 0$. This follows from \cite[Corollary 7.4, Theorem 3.3]{AW13}. 
\end{proof}

\begin{thm} Suppose that $X$ is a smooth rigid analytic space over $K$. Then every non-zero coadmissible $\w{\D}_X$-module has dimension at least $\dim X$.
\end{thm}

\begin{proof} Let $\M$ be a non-zero coadmissible $\w{\D}_X$-module. Since $\dim \M$ is defined locally we may assume that $X$ is affinoid and $\T(X)$ admits a smooth Lie lattice. Now every $K$-affinoid can be viewed as a closed analytic subset of a polydisc $Y=\Sp K\langle x_1,\ldots,x_N\rangle$ for $N$ sufficiently large. Let $\iota\colon X\to Y$ denote the closed embedding. By Theorem \ref{Kashiwaradim}, $d(\iota_+\M)= d(\M)+ N -\dim X$. Thus it suffices to show that $d(\iota_+\M)\geq N$, and the result follows from the Proposition. 
\end{proof}
This finishes the proof of Theorem \ref{introthmA}.(ii).

\section{Weakly holonomic $\w{\D}$-modules}
\subsection{Definition and basic properties}\label{weakhol}
\begin{defn}
A coadmissible $\w{\D}_X$-module $\M$ on a smooth rigid analytic $K$-variety $X$ is called \emph{weakly holonomic} if $d( \M)\leq \dim X$.
\end{defn}
We denote the full subcategory of $\C_X$ consisting of weakly holonomic $\w{\D}_X$-modules by $\C^{\mathrm{wh}}_X$.

\begin{prop}
Let
\begin{equation*}
0\to \M_1\to \M_2\to \M_3\to 0
\end{equation*}
be a short exact sequence of coadmissible $\w{\D}_X$-modules. Then $\M_2$ is weakly holonomic if and only if both $\M_1$ and $\M_3$ are weakly holonomic.
\end{prop}
\begin{proof}
As the dimension can be calculated locally, we can assume that $X=\Sp A$ is a smooth, connected affinoid $K$-space with $\T(X)$ admitting a smooth Lie lattice. Choosing an affine formal model $\A$ and a smooth $(R, \A)$-Lie lattice $\L$ in $\T(X)$, we can write $\w{\D}(X)=\varprojlim \widehat{U_{\A}(\pi^n\L)}_K$, where $\widehat{U_{\A}(\pi^n\L)}_K$ is Auslander--Gorenstein of dimension at most $2 \dim A$ for sufficiently large $n$ by Theorem \ref{CompleteEnvAlg}. As before, we abbreviate $D=\w{\D}(X)$ and $D_n=\hK{U(\pi^n\L)}$\\
If $0\to M_1\to M_2\to M_3\to 0$ is a short exact sequence of coadmissible $D$-modules, there exists an integer $m$ such that
\begin{equation*}
j_{D}(M_i)=j_{D_n}(D_n\otimes_D M_i) \ \text{for} \ i=1, 2, 3,\ n\geq m
\end{equation*}
by Lemma \ref{ReviewDim}, so the result follows from \cite[Proposition 4.5.(ii)]{Lev92} applied to
\begin{equation*}
0\to D_n\otimes M_1\to D_n\otimes M_2\to D_n\otimes M_3\to 0,
\end{equation*}
which is exact by flatness of $D_n$ over $D$.
\end{proof}
It follows immediately from the above that $\C^{\mathrm{wh}}_X$ is an abelian subcategory of $\C_X$.

\begin{lem}
Let $\iota: Y\to X$ be a closed embedding of smooth rigid analytic $K$-varieties. Then Kashiwara's equivalence (\cite[Theorem A]{DhatTwo}) restricts to an equivalence between $\C^{\mathrm{wh}}_Y$ and the category of weakly holonomic $\w{\D}_X$-modules supported on $Y$.
\end{lem}
\begin{proof}
This is a direct consequence of Lemma \ref{Kashiwaradim}.
\end{proof}
\subsection{Extensions}\label{extn}
Let $X$ be a smooth rigid analytic $K$-variety and let $\D_X$ denote the sheaf of algebraic (i.e. finite order) differential operators on $X$. In \cite{Mebkhout}, Mebkhout--Narvaez-Macarro developed a dimension theory on the category $\mathrm{coh}(\D_X)$ of coherent $\D_X$-modules by setting $d(\Loc M)=2\dim X-j_{\D(X)}(M)$ whenever $M$ is a finitely generated $\D(X)$-module on a smooth affinoid $K$-variety $X$ with free tangent sheaf. They also prove a version of Bernstein's inequality and define modules \emph{of minimal dimension} as the analogue of holonomicity.\\
\\
Recall from \cite[Lemma 4.14]{Bode} that there is an exact extension functor
\begin{align*}
E_X: & \ \mathrm{coh}(\D_X)\to \C_X\\
& \M\mapsto\w{\D}_X\otimes_{\D_X} \M.
\end{align*}
It follows from Theorem \ref{fflat} that $E_X$ is also faithful.

\begin{prop}
Let $\M$ be a coherent $\D$-module on a smooth, equidimensional rigid analytic $K$-space $X$. Then 
\begin{equation*}
d(\M)=d(E_X \M).
\end{equation*}
In particular, $E_X$ sends modules of minimal dimension to weakly holonomic $\w{\D}$-modules.
\end{prop}
\begin{proof}
Without loss of generality, $X=\Sp A$ is a smooth affinoid with free tangent sheaf. Write $D=\D(X)$ and $\w{D}=\w{\D}(X)$. If $M$ is a finitely generated $D$-module, then faithful flatness of $\w{D}$ over $D$ (Theorem \ref{fflat}) implies that for any $j$,
\begin{equation*}
\Ext^j_{\w{D}}(\w{D}\otimes_D M, \w{D})\cong \Ext^j_{D}(M, D)\otimes_D \w{D}=0 \ \text{if and only if} \ \Ext^j_D(M, D)=0.
\end{equation*}
In particular, $j_D(M)=j_{\w{D}}(\w{D}\otimes_D M)$, and the result follows.
\end{proof}
\begin{cor}
Let $\M$ be an integrable connection on a smooth rigid analytic $K$-space $X$. Then $\M$ is a weakly holonomic $\w{\D}_X$-module.
\end{cor}
\begin{proof}
Any integrable connection is a coherent $\D_X$-module of minimal dimension, so we are done by applying \cite[Proposition 6.2]{DhatTwo}.
\end{proof}
\subsection{Duality}
\begin{lem}
Let $X$ be a smooth rigid analytic $K$-space of dimension $d$. The functor $\sHom_{\O_X}(\Omega_X, \sExt^d_{\w{\D}_X}(-, \w{\D}_X)):\C_X\to \C_X^{\mathrm{op}}$ sends weakly holonomic $\w{\D}_X$-modules to weakly holonomic $\w{\D}_X$-modules.
\end{lem}
\begin{proof}
As $X$ is equidimensional, we can assume that $X$ is affinoid with $\T(X)$ admitting a smooth Lie lattice. Let $\M=\Loc M$ be a non-zero weakly holonomic $\w{\D}_X$-module, so that $j(M)=d$.\\
By Auslander's condition, $\Ext^d(M, \w{\D}(X))$ has grade $\geq d$ as a right $\w{\D}_X(X)$-module. Thus Proposition \ref{dimtheorysheaf} implies that 
\begin{equation*}
\sExt^d_{\w{\D_X}}(\mathcal{M}, \w{\D_X})\cong \Loc \Ext^d(M, \w{\D}_X(X))
\end{equation*}
is a coadmissible right $\w{\D}_X$-module of dimension $d$, and the result follows from Lemma \ref{leftright}.
\end{proof}
Similarly, Auslander's condition in conjunction with Bernstein's inequality forces $\sExt^i(\M, \w{\D})=0$ for any $\M\in \C^{\mathrm{wh}}_X$, $i\neq d$.\\
We define the \emph{duality functor} $\mathbb{D}$ on $\C^{\mathrm{wh}}_X$ by $\mathbb{D}=\sHom_{\O}(\Omega_X, \sExt^{d}_{\w{\D}_X}(-, \w{\D}_X))$.
\begin{prop}
There is a natural isomorphism of functors $\mathbb{D}^2\cong \id$.
\end{prop}
\begin{proof}
Let $X$ be a smooth affinoid of dimension $d$ with $\T(X)$ admitting a smooth Lie lattice $\L$. Write $D=\w{\D}(X)$ and $D_n=\hK{U(\pi^n\L)}$, so that $D=\varprojlim D_n$. By Theorem \ref{CompleteEnvAlg}, we can assume that $D_n$ is Auslander--Gorenstein of dimension at most $2d$ for each $n\geq 0$.\\
Let $M$ be a coadmissible $D$-module of grade $d$. By Lemma \ref{leftright} and \cite[Theorem 3.4]{DhatTwo}, we have
\begin{equation*}
\Hom_A(\Omega, \Ext^d_D(\Hom_A(\Omega, \Ext^d_D(M, D)), D))\cong \Ext^d_D(\Ext^d_D(M, D), D),
\end{equation*}
and now
\begin{align*}
\Ext^d(\Ext^d(M, D), D)&\cong\varprojlim \left(D_n\otimes_D \Ext^d_D(\Ext^d_D(M, D), D)\right)\\
& \cong \varprojlim \Ext^d_{D_n}(\Ext^d_D(M, D)\otimes_D D_n, D_n)\\
& \cong \varprojlim \Ext^d_{D_n}(\Ext^d_{D_n}(D_n\otimes_D M, D_n), D_n)
\end{align*}
by repeatedly applying Lemma \ref{ReviewDim}.\\
\\
Note that $\Ext^i_{D_n}(D_n\otimes_D M, D_n)=0$ for any $i\neq d$ by Lemma \ref{ReviewDim}, so
\begin{equation*}
\Ext^d(\Ext^d(D_n\otimes M, D_n), D_n)\cong D_n\otimes M
\end{equation*}
by \cite[Theorem 4]{Iwanaga}\footnote{In the reference, this is stated as a canonical isomorphism rather than a natural one, but it is clear that this is the natural morphism $M\to \mathrm{R}\Hom (\mathrm{R}\Hom(M, R), R)$ which becomes an isomorphism on this particular class of modules.}, and hence
\begin{equation*}
\Ext^d(\Ext^d(M, D), D)\cong \varprojlim (D_n\otimes_D M)\cong M
\end{equation*}
by coadmissibility of $M$.\\
Thus $\mathbb{D}^2(\Loc M)\cong \Loc M$ as required.
\end{proof}
\section{Examples}
\subsection{Infinite length and infinite-dimensional fibres}
We present an example of a weakly holonomic $\w{\D}$-module on the unit disc which is not of finite length and has infinite-dimensional fibres.\\
\\
Let $\theta_n(t)=\prod_{m=0}^n (1-\pi^mt)$, and consider the power series
\begin{equation*}
\theta(t)=\lim_{n\to \infty} \theta_n(t)=\prod_{m=0}^\infty(1-\pi^mt).
\end{equation*}
Note that $\theta(t)\in \w{K[t]}=\varprojlim K\langle \pi^nt \rangle$. We also note that for any $n\geq 0$, $(1-\pi^mt)$ is a unit in $K\langle \pi^nt\rangle$ for any $m>n$. Thus $\theta(t)=u_n\theta_n(t)$, where $u_n$ is a unit in $K\langle \pi^nt\rangle$.\\
\\
Let $X=\Sp K\langle x\rangle$ be the unit disc over $K$, and write $\partial\in \T(X)$ for the derivation $\mathrm{d}/\mathrm{d}x$. Let $D=\w{\D}(X)$ and set
\begin{equation*}
M=D/D\theta(\partial).
\end{equation*}
This is a coadmissible $\w{\D}(X)$-module, as it is finitely presented. \\
\\
Let $\A=R\langle x\rangle$, $\L=\A\cdot \partial \subset \T(X)$, and let $D_n=\widehat{U_{\A}(\pi^n\L)}_K$. By the considerations above, 
\begin{equation*}
D_n\otimes_D M\cong D_n/D_n\theta_n(\partial),
\end{equation*}
which is finitely generated over $K\langle x\rangle$. In particular, $d(M)=1$.\\
\\
But now note that for every $n\geq 0$, $M$ surjects onto $D/D\theta_n(\partial)$, which is a direct sum of $n+1$ integrable connections of rank $1$. In particular, $M$ can not be of finite length as a $D$-module.\\
\\
By the same argument, the fibre at zero $M/xM$ cannot be a finite-dimensional $K$-vector space. In particular, weakly holonomic $\w{\D}$-modules need not have finite-dimensional fibres, and weak holonomicity is generally not stable under pullback.
\subsection{Pushforward along an open embedding}\label{recallBitoun}
We now recall from \cite{Bitoun} that weak holonomicity is generally not stable under pushforward either. \\
For this, recall from \cite[Definition 13.1.1]{Kedlaya} that the \emph{type} of $\lambda \in K$ is the radius of convergence of the formal power series
\begin{equation*}
\sum_{i\geq 0, i\neq \lambda} \frac{x^i}{\lambda-i}.
\end{equation*}
Let $X=\Sp A$ be an affinoid $K$-space with free tangent sheaf, let $f\in A$ be non-constant and consider $j: U= \{f\neq 0\}\to X$. We call a left $\D(X)[f^{-1}]$-module $N$ a \emph{meromorphic connection on $X$ with singularities along $Z=X\setminus U$} if $N$ is free of finite rank over $A[f^{-1}]$.\\
Given $m\in N$, consider the ideal $I(m)\subseteq K[s]$ consisting of all polynomials $b(s)$ such that there exists some $P(s)\in \D(X)[s]$ satisfying
\begin{equation*}
P(s)f^{-s}m=b(s)f^{-s-1}m \ \forall s\in \mathbb{Z}.
\end{equation*}
By \cite[Th\'eor\`eme 3.1.1]{Mebkhout}, $I(m)$ is non-zero, and we call its monic generator the $b$-function of $m$.
\begin{thm}[{\cite[Theorems 1.2, 1.3]{Bitoun}}]
\begin{enumerate}[(i)]
\item Let $N$ be a meromorphic connection on $X$ with singularities along $Z$. Suppose that $N$ is generated as an $A[f^{-1}]$-module by $m_1, \dots, m_r$ such that all the roots of the associated $b$-functions $b_1, \dots, b_r$ in an algebraic closure of $K$ are of positive type. Then $N$ localizes to an integrable connection $\M$ on $U$ such that $j_*\M\in \C_X$.
\item Let $X=\Sp K\langle x\rangle$ and $j: U\to X$ for $U=X\setminus\{0\}$. Set $\M_{\lambda}=\O_Ux^{\lambda}$ for $\lambda \in K$. Then $j_*\M_{\lambda}\in \C_X$ if and only if $\lambda$ is of positive type. In particular, there exist integrable connections $\M$ on $U$ such that $j_*\M$ is not coadmissible. 
\end{enumerate}
\end{thm}

\section{Zariski open embeddings: The case of an algebraic snc divisor}
Let $j: U\to X$ be an embedding of a Zariski open subspace $U$ in a smooth rigid analytic $K$-space $X$. We show that at least in the case of the structure sheaf, pathologies as in Theorem \ref{recallBitoun}.(ii) do not occur.\\
\\
As the question is local, we can (and will) assume from now on that $X$ is a smooth affinoid with free tangent sheaf. We will first consider the case where $Z=X\setminus U$ is the analytification of a strict normal crossing divisor. In section 9, we will reduce the general case to this set-up by passing to a suitable resolution of singularities.

\subsection{Relative analytification}\label{relana}
We recall some results from \cite{Schoutens} regarding relative analytification. This was already considered by K\"{o}pf in \cite{Kopf}, but we have decided to refer to a source that is more readily available.\\
If $A$ is an affinoid $K$-algebra, there is a relative analytification functor $\mathbf{X}\mapsto \mathbf{X}^{\mathrm{an}}$ from schemes of finite type over $\Spec A$ (`$A$-schemes' for short) to rigid analytic $K$-varieties. The construction of $\mathbf{X}^{\mathrm{an}}$ is a straightforward generalization of the analytification procedure for $A=K$.\\
There is a natural morphism
\begin{equation*}
\eta_{\mathbf{X}}: \mathbf{X}^{\mathrm{an}}\to \mathbf{X}
\end{equation*}
of locally $G$-ringed spaces, satisfying the usual universal property: any morphism $Y\to \mathbf{X}$ of locally $G$-ringed spaces, where $Y$ is a rigid analytic $K$-variety, factors uniquely through $\eta_{\mathbf{X}}$ (\cite[Definition 1.1.1]{Schoutens}).
\begin{prop} Let $A$ and $B$ be affinoid $K$-algebras. 
\begin{enumerate}[(i)]
\item If $\mathbf{X}$ is both an $A$-scheme and a $B$-scheme, then the analytification of $\mathbf{X}$ relative to $A$ is isomorphic to the analytification of $\mathbf{X}$ relative to $B$. In particular, we can speak of `the' analytification of $\mathbf{X}$.
\item The analytification of $\Spec A$ is $\Sp A$.
\item If $\mathbf{Z}$ is a closed subscheme of an $A$-scheme $\mathbf{X}$ given by the vanishing ideal $\mathcal{I}$, then $Z:=\eta^{-1}(\mathbf{Z})$ is the closed analytic subset of $\mathbf{X}^{\mathrm{an}}$ defined by $\eta^*\mathcal{I}$. Moreover, $Z\cong \mathbf{Z}^{\mathrm{an}}$.
\item If $\mathbf{Y}$ is an open subscheme of an $A$-scheme $\mathbf{X}$, then $\mathbf{Y}^{\mathrm{an}}$ can be identified with $\eta^{-1}(\mathbf{Y})$, a Zariski open subspace of $\mathbf{X}^{\mathrm{an}}$.
\item If $\mathbf{X}_i$, $\mathbf{X}$ are $A$-schemes such that the $\mathbf{X}_i$ form a covering of $\mathbf{X}$, then the $\mathbf{X}_i^{\mathrm{an}}$ form an admissible covering of $\mathbf{X}^{\mathrm{an}}$.
\item If $\mathbf{X}\to\mathbf{S}, \mathbf{Y}\to \mathbf{S}$ are morphisms of $A$-schemes, then $(\mathbf{X}\times_{\mathbf{S}}\mathbf{Y})^{\mathrm{an}}\cong \mathbf{X}^{\mathrm{an}}\times_{\mathbf{S}^{\mathrm{an}}}\mathbf{Y}^{\mathrm{an}}$. 
\end{enumerate}
\end{prop}
\begin{proof}
(i) is clear from the universal property, (ii) is \cite[Example 1.3.2.(1)]{Schoutens}. (iii) is \cite[Corollary 2.1.3]{Schoutens}. Concerning (iv), it follows easily from the universal property that $\mathbf{Y}^{\mathrm{an}}\cong \eta^{-1}(\mathbf{Y})$, which is an admissible open subspace of $\mathbf{X}^{\mathrm{an}}$ by continuity of $\eta$. By (iii), it is then even Zariski open. (v) is a consequence of (iv) and the fact that $\eta$ is a morphism of locally $G$-ringed spaces. (vi) follows again from the universal property.
\end{proof}
\subsection{Algebraic snc divisors}\label{algsnc}
\begin{defn}
A closed subscheme $\mathbf{Z}$ of a locally Noetherian regular scheme $\mathbf{X}$ is called a \emph{strict normal crossing (snc) divisor} if
\begin{enumerate}[(i)]
\item $\mathbf{Z}$ is defined as the vanishing of an invertible ideal sheaf $\mathcal{I}$.
\item for each $x\in \mathbf{Z}$, there exists a regular set of local parameters $x_1, \dots, x_d\in \O_{\mathbf{X}, x}$ such that the ideal $\mathcal{I}_x\subseteq \O_{\mathbf{X}, x}$ is generated by $\prod_{i=1}^r x_i$ for some $1\leq r\leq d$.
\end{enumerate} 
\end{defn}
\begin{defn}
A closed analytic subset $Z$ of a rigid analytic $K$-variety $X$ is called an \emph{algebraic snc divisor} if there exists an affinoid $K$-algebra $B$, a regular scheme $\mathbf{X}$ that is also a $B$-scheme with $\mathbf{X}^{\mathrm{an}}\cong X$, and an snc divisor $\mathbf{Z}$ of $\mathbf{X}$ such that $\mathbf{Z}^{\mathrm{an}}=Z$ as in Proposition \ref{relana}.(iii). 
\end{defn}
We now wish to lift the snc condition from the stalk level to a condition on the level of admissible coverings.\\
\\
Note that if $X=\Sp A$ is a smooth affinoid, then $\mathbf{X}=\Spec A$ is a Noetherian regular scheme by \cite[Proposition 7.3.2/8]{BGR}.
\begin{lem}
Let $X=\Sp A$ be a smooth affinoid $K$-space, and let $\mathbf{Z}\subseteq \Spec A=\mathbf{X}$ be an snc divisor. Then there exists an admissible covering of $X$ by affinoid subdomains $X_i=\Sp A_i$ with the following property: for each $i$ with $\mathbf{Z}^{\mathrm{an}}\cap X_i\neq \emptyset$, there exist $x_{i1}, \dots, x_{id}\in A_i$ such that
\begin{enumerate}[(i)]
\item the elements $\mathrm{d}x_{i1}, \dots, \mathrm{d}x_{id}$ form a free generating set of $\Omega^1(X_i)$, and
\item the subvariety $\mathbf{Z}^{\mathrm{an}}\cap X_i$ is given as the vanishing set of $\prod_{j=1}^r x_{ij}$ for some $1\leq r\leq d$.
\end{enumerate}
\end{lem}
\begin{proof}
Without loss of generality, we can assume that $X$ (and hence $\mathbf{X}$) is connected.\\
Let $x\in \mathbf{Z}$ be a closed point in $\Spec A$, and let $x_1, \dots, x_d\in \O_{\mathbf{X}, x}$ be a regular set of local parameters as in Definition \ref{algsnc}. By definition, there exists some Zariski open affine subscheme $U_x=\Spec B(x)\subset \mathbf{X}$ containing $x$ such that $x_1, \dots, x_d$ are defined on $U_x$, and $\mathbf{Z}\cap U_x=\{\prod_{j=1}^r x_j=0\}$.\\
\\
By smoothness, $\Omega^1(X)$ is a finitely generated projective $A$-module, so there exists $f_1, \dots, f_n\in A$ such that
\begin{equation*}
A_{f_k}\otimes_A \Omega^1(X)
\end{equation*}  
is a free $A_{f_k}$-module for each $k$, with free generating set $s_{k1}, \dots, s_{kd}$.\\
Fix $k$ such that $x\in \Spec A_{f_k}$. Then $B(x)_{f_k}\otimes_A \Omega^1(X)$ is freely generated by $s_{k1}, \dots, s_{kd}$, so there exist $M=(m_{ij})\in \mathrm{Mat}_{d\times d}(B(x)_{f_k})$ such that
\begin{equation*}
\mathrm{d}x_i=\sum_j m_{ij} s_{kj}.
\end{equation*}
As $x\in \mathbf{X}$ is closed, it corresponds to a unique point in $X=\Sp A$, which we also denote by $x$. 
As the $x_i$ form a regular system of local parameters in $\O_{\mathbf{X}, x}$, they also define a regular system of local parameters in the local ring $\O_{X, x}$ by \cite[Propositions 4.1/1, 4.1/2]{Bosch}, so that $\{\mathrm{d}x_i\}$ form a free generating set in $\Omega^1_{X, x}=\O_{X, x}\otimes_A \Omega^1(X)$. Thus $M$ becomes invertible as a matrix over $\O_{X, x}$ and hence over $\O_{\mathbf{X}, x}$, as $\O_{\mathbf{X}, x}\to \O_{X, x}$ is injective by \cite[Proposition 4.1/2]{Bosch}. Therefore $x$ is contained in a Zariski open affine subscheme $\Spec C(x)$ of $\mathbf{X}$ with the property that
\begin{enumerate}[(i)]
\item $x_1, \dots, x_d\in C(x)$, and $\mathrm{d}x_i$ form a free generating set of $C(x)\otimes_A \Omega^1(X)$.
\item $\mathbf{Z}\cap \Spec C(x)=\{\prod_{j=1}^r x_j=0\}$ for some $1\leq r\leq d$.
\end{enumerate}
As the $\Spec C(x)$ for varying $x\in \mathbf{Z}$ together with $\mathbf{X}\setminus \mathbf{Z}$ form a Zariski covering of $\mathbf{X}$, it follows from Proposition \ref{relana}.(v) that their analytifications form an admissible covering of $X$, and any refinement of this covering by affinoid subdomains has the desired property. 
\end{proof}
\begin{prop}
Let $Z$ be an algebraic snc divisor of a smooth rigid analytic $K$-variety $X$. Then there exists an admissible covering of $X$ by affinoids $X_i=\Sp A_i$ with the following property: for each $i$ with $Z\cap X_i\neq \emptyset$, there exist $x_{i1}, \dots, x_{id}\in A_i$ such that
\begin{enumerate}[(i)]
\item the elements $\mathrm{d}x_{i1}, \dots, \mathrm{d}x_{id}$ form a free generating set of $\Omega^1(X_i)$, and
\item the subvariety $Z\cap X_i$ is given as the vanishing set of $\prod_{j=1}^r x_{ij}$ for some $1\leq r\leq d$.
\end{enumerate}
\end{prop}
\begin{proof}
By the above Lemma, it is sufficient to show that $X$ admits an admissible covering by affinoid subspaces $Y_i=\Sp B_i$ such that $Z\cap Y_i$ is obtained as the analytification of an snc divisor on $\Spec B_i$. \\
\\
Let $B$ be an affinoid $K$-algebra, $\mathbf{X}$ a $B$-scheme with snc divisor $\mathbf{Z}$ such that $\mathbf{X}^{\mathrm{an}}\cong X$, $\mathbf{Z}^{\mathrm{an}}\cong Z$. Let $(\mathbf{U}_i)$ be an affine covering of $\mathbf{X}$ such that $\mathbf{Z}\cap \mathbf{U}_i$ is given by the vanishing of a single element $f_i\in B_i$, where $\mathbf{U}_i=\Spec B_i$.\\
\\
By Proposition \ref{relana}.(v), $(\mathbf{U}_i^{\mathrm{an}})_i$ is an admissible covering of $X$. Recall from \cite[Construction 1.2.1]{Schoutens} that the analytification $\mathbf{U}_i^{\mathrm{an}}$ is constructed as the union of $\Sp B_{i, n}$ for various Banach completions $B_{i, n}$ of $B_i$. Thus by definition, $(\Sp B_{i, n})_{i, n}$ is also an admissible covering of $X$.\\
There are natural morphisms of schemes
\begin{equation*}
\phi_{i, n}: \Spec B_{i, n}\to \Spec B_i=\mathbf{U}_i\subseteq \mathbf{X}.
\end{equation*}
Now $\Spec B_{i, n}$ is a regular Noetherian scheme by \cite[Proposition 7.3.2/8]{BGR}. Let $\mathbf{Z}_{i, n}$ denote the vanishing set of $\phi_{i, n}^{\sharp}(f_i)\in B_{i, n}$. Then $\mathbf{Z}_{i, n}$ is an snc divisor, as for any $x\in \mathbf{Z}_{i, n}$, a regular set of local parameters in $\O_{\mathbf{U}_i,\phi_{i, n}(x)}$ gives a regular set of local parameters of $\O_{\Spec B_{i, n}, x}$ by \cite[Claim 1.2.6]{Schoutens}.\\
\\
Since $\mathbf{Z}_{i, n}^{\mathrm{an}}=\{f_i=0\}\subset \Sp B_{i, n}$ by Proposition \ref{relana}.(iii), it follows that $\mathbf{Z}_{i, n}^{\mathrm{an}}=\mathbf{Z}^{\mathrm{an}}\cap \Sp B_{i, n}=Z\cap \Sp B_{i, n}$, as required.
\end{proof}

\subsection{\bf{Proposition}}\label{snc}
Let $X$ be a smooth rigid analytic $K$-variety. Let $Z$ be an algebraic snc divisor on $X$, and let $j: U\to X$ be its complement. Then $j_*(\M|_U)$ is a coadmissible $\w \D_X$-module for any integrable connection $\M$ on $X$.

\begin{proof}
By Proposition \ref{algsnc}, we can assume that $X\cong \Sp A$ is a smooth affinoid with free tangent sheaf, and also that there exist elements $x_1, \dots, x_d\in A$ such that $\{\mathrm{d}x_i\}$ form a free generating set in $\Omega^1_X$ and $Z=\{\prod_{i=1}^r x_i=0\}$ for some $1\leq r\leq d$.\\
We denote by $\{\partial_i\}$ the derivations forming the basis in $\T(X)$ dual to $\{\mathrm{d}x_i\}$, so that 
\begin{equation*}
\partial_i(x_j)=\delta_{ij}, \ [\partial_i, \partial_j]=0
\end{equation*} 
for all $i, j$. We write $f=\prod_{i=1}^r x_i$.\\
\\
Let $M=\M(X)$, which we can assume to be a free $A$-module of finite rank, and let $m\in M$. We will now show that all roots of the $b$-function of $m$, viewed as an element in $M[f^{-1}]$, are integers, so that we can apply Theorem \ref{recallBitoun}.(i).\\
\\
We write $\partial_i^{[k]}=\frac{\partial_i^k}{k!}$ for any non-negative integer $k$. By Noetherianity, there exists some natural number $n$ such that $\partial_i^{[n]}\cdot m$ is contained in the $A$-submodule generated by $m, \partial_i \cdot m, \dots, \partial_i^{[n-1]}\cdot m$, so that there exists a monic polynomial $\sum a_{i, j}x^j\in A[x]$ of degree $n$ satisfying
\begin{equation*}
\sum_{j=0}^n a_{i, j}\partial_i^{[j]}\cdot m=0.
\end{equation*}
Now in $\D(X)[f^{-1}]$, one obtains as usual (see e.g. \cite[equation 2.0.2]{Berthelot})
\begin{align*}
\partial_i^{[j]}f^s&=\sum_{k=0}^j \partial_i^{[j-k]}(f^s)\partial_i^{[k]} \\
&=\sum_{k=0}^j \binom{s}{j-k} x_i^{k-j} f^s \partial_i^{[k]} \ \forall s\in \mathbb{Z},
\end{align*}
which can be verified by straightforward induction on $j$.\\
By writing $m=f^sf^{-s} m$, we thus have in $M[f^{-1}]$ the following equations for any integer $s$:
\begin{align*}
0&= x_i^{n-1}f^{-s}\sum_{j=0}^n a_{i, j}\partial_i^{[j]}\cdot m\\
&= \sum_{j=0}^n x_i^{n-j} a_{i, j} \left( x_i^{j-1} f^{-s} \partial_i^{[j]}\cdot m\right)\\
&= \sum_{j=0}^n x_i^{n-j}a_{i, j} \left(\sum_{k=0}^j \binom{s}{j-k}x_i^{k-1} \partial_i^{[k]}\cdot f^{-s}m\right)\\
&= \sum_{k=0}^n \sum_{j=k}^n \binom{s}{j-k}x_i^{n-j+k-1} a_{i, j} \partial_i^{[k]}\cdot f^{-s}m.
\end{align*}
As $a_{i, n}=1$, we obtain
\begin{equation*}
\binom{s}{n}x_i^{-1}f^{-s}m=-\sum\nolimits'\binom{s}{j-k}x_i^{n-j+k-1}a_{i, j}\partial_i^{[k]}\cdot f^{-s}m,
\end{equation*}
where $\sum'$ denotes the sum over all pairs $(j, k)$, $0\leq k\leq j\leq n$, $(j, k)\neq (n, 0)$.\\
\\
Setting $P_i(s)=-\sum' \binom{s}{j-k}x_i^{n-j+k-1}a_{i, j} \partial_i^{[k]}\in \mathcal{D}(X)[s]$, we note that $P_i(s)$ is contained in the subring $A[\partial_i, s]$ and hence commutes with $x_j$ for any $j\neq i$. We thus conclude from the above that
\begin{equation*}
P_i(s) \cdot x_1^{-1}\dots x_{i-1}^{-1} f^{-s}m=\binom{s}{n} x_1^{-1} \dots x_i^{-1} f^{-s}m
\end{equation*}
for any $1\leq i\leq r$, and by induction
\begin{equation*}
\left(\prod_{i=1}^r P_i(s)\right)\cdot f^{-s}m=\binom{s}{n}^r f^{-s-1}m.
\end{equation*}
In particular, $\binom{s}{n}^r\in I(m)$. Thus the $b$-function of $m$ is a factor of $\binom{s}{n}^r$, and all its roots are integers. As integers are of positive type by \cite[Proposition 13.1.5]{Kedlaya}, applying Theorem \ref{recallBitoun}.(i) proves the result.
\end{proof}
\section{Zariski open embeddings: The general case}
\subsection{Cohomology on hyperplane complements}\label{weaklyStein}
We introduce the following notation: if $X=\Sp A$ and $f\in A$ is non-constant, we denote by $X_f$ the admissible open subspace given by the non-vanishing of $f$. Note that $\O(X_f)\neq A_f$, but rather
\begin{equation*}
\O(X_f)=\varprojlim A\langle \pi^nf^{-1}\rangle.
\end{equation*}
The following is a partial generalisation of \cite[Satz 2.4.2]{Kiehl}.
\begin{prop}
Let $X=\Sp A$ be smooth and let $f\in A$ be non-constant. If $\M$ is a coadmissible $\w{\D}$-module on $X_f$ then $\mathrm{H}^i(X_f, \M)=0$ for every $i>0$.
\end{prop}
\begin{proof}
Let $\A\subset A$ be an affine formal model and let $\L\subset \T(X)$ be an $\A$-Lie lattice. Then after rescaling $f$ if necessary, we can assume without loss of generality that $f\in \A$, so that $U_m:=X(\pi^mf^{-1})$ is $\pi^n\L$-accessible for any $n\geq m$.\\
We write $D=\w{\D}(X_f)$, $D_n=\hK{U(\L_n)}$, where $\L_n$ is the image of $\A\langle \pi^nf^{-1}\rangle \otimes_{\A} \pi^n\L$ inside $\T(U_n)$. Let $\D_n$ be the sheaf of algebras on the site of $\L_n$-admissible subspaces of $U_n$ given by $\O\h{\otimes}_{\O(U_n)} D_n$. As the maps $D_{n+1}\to \D_{n+1}(U_n)\to D_n$ are flat by \cite[Theorem 4.10]{Bode} and \cite[Theorem 6.6]{DhatOne}, $D=\varprojlim D_n$ exhibits $D$ as a Fr\'echet--Stein algebra, and if $\M$ is a coadmissible $\w{\D}_{X_f}$-module then $\M(X_f)=\varprojlim M_n$ for $M_n=D_n\otimes_{\w{\D}(U_n)}\M(U_n)$ makes $\M(X_f)$ a coadmissible $D$-module.\\
\\
Let $\M_n$ be the sheaf $\Loc M_n$ on the site of $\L_n$-admissible subspaces of $U_n$, a coherent $\D_n$-module (see \cite[\textsection 5.1]{DhatOne}). Let $\mathfrak{U}=(U_i)$ and let $\mathfrak{U}_n=\{U_1, \dots, U_n\}$. Note that by \cite[Proposition 9.5]{DhatOne}
\begin{equation*}
\mathrm{H}^i(X_f, \M)\cong\check{\mathrm{H}}^i(\mathfrak{U}, \M)
\end{equation*} 
for any $i$.\\
Consider the complexes $C(n)^{\bullet}=\check{C}^{\bullet}(\mathfrak{U}_n, \M_n)$ with natural morphisms of complexes $C(n+1)^\bullet\to C(n)^\bullet$ induced by restriction. By \cite[Theorem 4.16]{Bode},
\begin{equation*}
\check{\mathrm{H}}^i(\mathfrak{U}_n, \M_n)=0 \ \forall \ i>0, \ \forall n,
\end{equation*}
and $\check{\mathrm{H}}^0(\mathfrak{U}_{n+1}, \M_{n+1})=M_{n+1}\to \check{\mathrm{H}}^0(\mathfrak{U}_n, \M_n)=M_n$ is a continuous morphism of Banach spaces with dense image for each $n$ by \cite[\textsection 3, Theorem A]{ST}.\\
\\
Thus by \cite[Proposition 13.2.3, Remarques 13.2.4]{EGA3}, 
\begin{equation*}
\mathrm{H}^i(\varprojlim C(n)^\bullet)\cong \varprojlim \mathrm{H}^i(C(n)^\bullet)
\end{equation*}
for each $i$.\\
It hence remains to show that $\varprojlim C(n)^\bullet\cong \check{C}^\bullet(\mathfrak{U}, \M)$. But
\begin{equation*}
\varprojlim_n C(n)^\bullet=\varprojlim_{\substack{m, n\\m\geq n}} \check{C}^\bullet(\mathfrak{U}_n, \M_m)\cong \varprojlim_n \check{C}^\bullet(\mathfrak{U}_n, \M)\cong \check{C}^\bullet(\mathfrak{U}, \M),
\end{equation*} 
as required.
\end{proof}

\begin{cor}
Let $X$ be a smooth affinoid $K$-space and let $j:U\to X$ be a Zariski open embedding. Let $\mathfrak{U}=(U_i)$ be an admissible covering of $U$ with the following property: for each $i$, there exists a smooth morphism of affinoids $V_i=\Sp A_i\to X$ fitting into a commutative diagram
\begin{equation*}
\begin{xy}
\xymatrix{
U_i\ar[r] \ar[rd]& V_i\ar[d]\\ & X
}
\end{xy}
\end{equation*}
which identifies $U_i$ with $(V_i)_{f_i}$ for some non-constant $f_i\in A_i$.\\
If $\M$ is a coadmissible $\w{\D}$-module on $U$ then $\mathrm{R}^ij_*\M(X)\cong \check{\mathrm{H}}^i(\mathfrak{U}, \M)$ for any $i\geq 0$.
\end{cor}
\begin{proof}
Note that $V_{i_1}\times_X \dots \times_X V_{i_r}$ is a smooth affinoid such that $U_{i_1}\cap \dots \cap U_{i_r}$ can be identified with $(V_{i_1}\times_X \dots \times_X V_{i_r})_{f_{i_1}\dots f_{i_r}}$, so it follows from the Proposition that for $i\geq 1$,
\begin{equation*}
\mathrm{H}^i(U_{i_1}\cap \dots \cap U_{i_r}, \M)=0
\end{equation*}
for any $i_1, \dots, i_r$. Thus the result follows from \cite[Tag 03F7]{stacks}.
\end{proof}
\subsection{Completed tensor products and sections of $\w{\D}$}\label{ctp}
In this subsection, we use some basic results from \cite{Bitoun} about completed tensor products to describe sections of $\w{\D}$ over Zariski open subspaces.
\begin{lem}
Let $X=\Sp A$ be an affinoid $K$-space with free tangent sheaf. Then the functor $-\h{\otimes}_A \w{\D}(X)$ is strict exact on Fr\'echet $A$-modules.
\end{lem}
\begin{proof}
As a left $A$-module, $\w{\D}(X)$ is isomorphic to $\O(X\times \mathbb{A}^{d, \mathrm{an}})$, where $d$ is the rank of $\T_X$. The result now follows from \cite[Propositions 1.2.2 and 1.2.6.(2)]{Kisin99}.
\end{proof}
\begin{prop}
Let $X=\Sp A$ be an affinoid $K$-space with free tangent sheaf, and let $U$ be a Zariski open subspace. Then the natural morphism
\begin{equation*}
\O(U)\h{\otimes}_A \w{\D}(X)\to \w{\D}(U)
\end{equation*} 
is an isomorphism of locally convex $\O(U)$-modules.
\end{prop}
\begin{proof}
First let $U=X_f$ for some $f\in A$. Let $\A$ be an affine formal model, and let $\L$ be an $\A$-Lie lattice in $\T(X)$. Then \cite[\textsection 3.3, Theorem 3.5]{DhatOne} gives rise to sheaves of $K$-algebras $\D_n$ on the site of $\pi^n\L$-admissible affinoid subdomains of $X$ such that for any $\pi^n\L$-admissible affinoid subdomain $Y=\Sp B\subset X$, we have $\w{\D}(Y)\cong \varprojlim_{m\geq n} \D_m(Y)$ and $\D_n(Y)\cong B\h{\otimes}_A \D_n(X)$. Without loss of generality, $f\in \A$, so that $U_n:=X(\pi^nf^{-1})$ is $\pi^n\L$-admissible. Now
\begin{equation*}
\w{\D}(U)\cong \varprojlim \w{\D}(U_n)\cong \varprojlim \D_n(U_n)\cong \varprojlim \left(\O(U_n)\h{\otimes}_A \D_n(X)\right), 
\end{equation*}
so that the result follows from \cite[Lemma A.2.(iv)]{Bitoun}.\\ 
\\
If $U$ is the complement of $V(f_1, \dots, f_r)$, we consider the Cech complex $\check{C}^\bullet((X_{f_i}), \O)$. Applying $-\h{\otimes}_A \w{\D}(X)$ and invoking the Lemma then finishes the proof in general.
\end{proof}
We also emphasize that by \cite[Corollary A.6]{Bitoun} the coadmissible tensor product $\w{\otimes}$ defined in \cite[section 7.3]{DhatOne} agrees with the completed tensor product $\h{\otimes}$ with respect to the canonical Fr\'echet structures on coadmissible modules.

\subsection{An algebraic analogue}\label{classicalpushforward}
In order to make the arguments below more accessible, we recall in some detail the following basic property of coherent $\D$-modules on algebraic varieties.
\begin{prop}
Let $k$ be an algebraically closed field of characteristic zero, and consider a diagram
\begin{equation*}
\begin{xy}
\xymatrix{
U \ar[r]^{j'} \ar[rd]^j & X'\ar[d]^{\rho} \ar[r]^{\iota}& \mathbb{P}^n\times X\ar[dl]^{\mathrm{pr}}\\
& X
}
\end{xy}
\end{equation*}
of smooth algebraic varieties over $k$, where $j$ and $j'$ are open embeddings, $\iota$ is a closed embedding, and $\mathrm{pr}$ is the natural projection. Assume further that $j'$ is affine. \\
If $\mathcal{M}$ is a coherent $\mathcal{D}_U$-module such that $j'_*\mathcal{M}$ is a coherent $\D_{X'}$-module, then $\mathrm{R}^ij_*\M$ is a coherent $\D_X$-module for every $i\geq 0$.
\end{prop}
\begin{proof}
For any morphism $f: Y'\to Y$ of smooth algebraic varieties over $k$, one can define (\cite[p.40]{HTT}) the $\D$-module pushforward functor  
\begin{align*}
f_+: \mathrm{D}^b(\D_{Y'})\to \mathrm{D}^b(\D_Y).
\end{align*} 
For right modules, this takes the form
\begin{equation*}
\M\mapsto \mathrm{R}f_*(\M\otimes^{\mathbb{L}}_{\D_{Y'}} \D_{Y'\to Y}),
\end{equation*}
where $\D_{Y'\to Y}$ is the transfer bimodule $f^*\D_Y$. For left modules, we compose the above with suitable side-changing operations.\\
\\
By \cite[Proposition 1.5.21, Example 1.5.22]{HTT} and the assumption on $j'$, we have
\begin{equation*}
\mathrm{R}j_*\M=j_+\M=\mathrm{pr}_+\iota_+j'_+\M=\mathrm{pr}_+\iota_+j'_*\M.
\end{equation*}
Write $P=\mathbb{P}^n\times X$.\\
Now $\iota_+$ preserves coherence by Kashiwara's equivalence (\cite[Theorem 1.6.1]{HTT}), and $\mathrm{pr}_+$ sends $\mathrm{D}^b_c(\D_P)$ to $\mathrm{D}^b_c(\D_X)$ by \cite[Theorem 2.5.1]{HTT}. This immediately proves the proposition.
\end{proof}
More explicitly, if $\mathcal{M}$ is a coherent right $\D_U$-module, we can write
\begin{equation*}
\mathrm{R}j_*\M=\mathrm{R}\mathrm{pr}_*(\iota_+j'_*\M\otimes^{\mathbb{L}}_{\D_P}\D_{P\to X}),
\end{equation*}
where $\D_{P\to X}$ is the $(\D_P, \mathrm{pr}^{-1}\D_X)$-bimodule $\mathrm{pr}^*\D_X=\sU_{\O_P}(\mathrm{pr}^*\T_X)$, the enveloping sheaf of the Lie algebroid $\mathrm{pr}^*\T_X$.\\
\\
To prove the general case of Theorem \ref{introthmB}, we will use an embedded resolution of singularities $(X', U)\to (X, U)$ to reduce to the case of an algebraic snc divisor, and an analogue of the above Proposition in order to descend from $X'$ to $X$ again. While a theory of general $\w{\D}$-module pushforwards has not been established yet, it turns out that we can make all computations explicitly in our situation.
\subsection{Pushforward of the structure sheaf on a Zariski open subspace} \label{gencase}
Let $j: U\to X$ be an arbitrary Zariski open embedding of smooth rigid analytic $K$-spaces. As always, we can assume that $X=\Sp A$ is affinoid with free tangent sheaf. \\
Writing $Z=\Sp B$ for the complement of $U$ with its reduced subvariety structure, we obtain a closed embedding $\Spec B\to \Spec A$ of schemes over $\mathbb{Q}$.
We recall the following result from \cite{Temkin}.
\begin{thm}[{\cite[Theorem 1.1.11]{Temkin}}]
\label{ers}
Let $\mathbf{X}$ be a quasi-excellent Noetherian scheme over $\mathbb{Q}$, and let $\mathbf{Z}$ be a closed subscheme of $\mathbf{X}$. Then there exists a regular scheme $\mathbf{X}'$ and a morphism of schemes $\rho: \mathbf{X}'\to \mathbf{X}$ which is a sequence of blow-ups with centres contained in $\mathbf{Z}\cup \mathbf{X}_{\mathrm{sing}}$ such that $\mathbf{Z}\times_{\mathbf{X}} \mathbf{X}'$ is an snc divisor in $\mathbf{X}'$. 
\end{thm}

Let $\mathbf{X}=\Spec A$, $\mathbf{Z}=\Spec B$. By \cite[Satz 3.3.3]{BKKN}, $\mathbf{X}$ is an excellent Noetherian scheme, so we can apply Theorem \ref{ers}. As $\mathbf{X}$ is also regular by \cite[Proposition 7.3.2/8]{BGR}, the centres of the blow-ups in the Theorem are contained in $\mathbf{Z}$. In particular, the Theorem provides us with a morphism $\rho: \mathbf{X}'\to \mathbf{X}$ which is an isomorphism away from $\mathbf{Z}$, so that we have the following commutative diagram
\begin{equation*}
\begin{xy}
\xymatrix{
\mathbf{U}\ar[r]^{j'}\ar[rd]_j& \mathbf{X}'\ar[d]^{\rho}\\
& \mathbf{X},
}
\end{xy}
\end{equation*}
where $j:\mathbf{U}\to \mathbf{X}$ is the complement of $\mathbf{Z}$ inside $\mathbf{X}$, and $j'$ realizes $\mathbf{U}$ as the complement of an snc divisor in $\mathbf{X}'$. Finally, we note that $\rho$ is projective, as it is a sequence of blow-ups, so that we obtain a factorization of $\rho$ as a closed immersion and a projection:
\begin{equation*}
\begin{xy}
\xymatrix{
\mathbf{U} \ar[r]^{j'} \ar[rd]_j & \mathbf{X}'\ar[r]^{\iota} \ar[d]& \mathbb{P}^n\times \mathbf{X}\ar[ld]^{\mathrm{pr}}\\
& \mathbf{X}
}
\end{xy}
\end{equation*}

By construction, all schemes in this diagram are $A$-schemes, and applying analytification, we obtain the commutative diagram of rigid analytic $K$-spaces
\begin{equation*}\label{ERS}
\begin{xy}
\xymatrix{
U \ar[r]^{j'} \ar[rd]_j& X'\ar[r]^{\iota}\ar[d]^{\rho}& \mathbb{P}^{n, \mathrm{an}}\times X\ar[ld]^{\mathrm{pr}}\\
& X
}
\end{xy}\tag{$*$}
\end{equation*}
where 
\begin{enumerate}[(i)]
\item $j$ and $j'$ are Zariski open embeddings by Proposition \ref{relana}.(iv), and $(\mathbf{U})^{\mathrm{an}}=(\mathbf{X}\setminus \mathbf{Z})^{\mathrm{an}}=X\setminus Z=U$ by Proposition \ref{relana}.(iii).
\item $X'=\mathbf{X}'^{\mathrm{an}}$ is a smooth rigid analytic $K$-space by \cite[Corollary 1.3.5]{Schoutens} and \cite[Lemma 2.8]{BLR3}.
\item $X'\setminus U$ is an algebraic snc divisor in $\mathbf{X}'$ by construction.
\item $\iota$ is a closed immersion of rigid analytic spaces by Proposition \ref{relana}.(iii). 
\item $(\mathbb{P}^n\times \mathbf{X})^{\mathrm{an}}\cong \mathbb{P}^{n, \mathrm{an}}\times (\mathbf{X})^{\mathrm{an}}=\mathbb{P}^{n, \mathrm{an}}\times X$ by Proposition \ref{relana}.(ii) and (vi).
\end{enumerate}
We are now in a position to discuss the rigid analytic analogue of Proposition \ref{classicalpushforward}. 
\begin{prop}
Let $\M$ be a coadmissible $\w{\D}$-module on $U$. If $j'_*\M$ is a coadmissible $\w{\D}_{X'}$-module, then $\mathrm{R}^ij_*\M$ is a coadmissible $\w{\D}_X$-module for every $i\geq 0$.
\end{prop}
\begin{proof}
We prove the proposition for right coadmissible $\w{\D}$-modules. The left module analogue follows by applying the sidechanging operators.\\
We write $P=\mathbb{P}^{n, \mathrm{an}}\times X$. By \cite[\textsection 5.4]{DhatTwo}, there exists a right coadmissible $\w{\D}_P$-module $\iota_+(j'_*\M)$ such that the following holds: for each admissible open affinoid $V$ of $P$ with the property that $ X'\cap V$ has a free tangent sheaf, we have
\begin{equation*}
\iota_+(j'_*\M)(V)=(j'_*\M)(X'\cap V)\underset{\w{\D}(X'\cap V)}{\h{\otimes}}\left(\O(X'\cap V)\underset{\O(V)}{\h{\otimes}}\w{\D}(V)\right).
\end{equation*}
We now note that $\mathrm{pr}^*\T_X$ is a Lie algebroid on $P$ which is free as an $\O_P$-module, and is naturally a direct summand of $\T_P$. This makes 
\begin{equation*}
\w{\sU(\mathrm{pr}^*\T_X)}\cong\O_P\underset{\mathrm{pr}^{-1}\O_X}{\h{\otimes}}\mathrm{pr}^{-1}\w{\D}_X
\end{equation*}
a $(\w{\D}_P, \mathrm{pr}^{-1}\w{\D}_X)$-bimodule. \\
\\
We can thus form the right coadmissible $\w{\sU(\mathrm{pr}^*\T_X)}$-module
\begin{equation*}
\N:=\iota_+(j'_*\M)\underset{\w{\D}_P}{\w{\otimes}} \w{\sU(\mathrm{pr}^*\T_X)}.
\end{equation*}
By \cite[Theorem 6.11]{Bodeproper}, $\mathrm{R}^i\mathrm{pr}_*\N$ is coadmissible over $\mathrm{pr}_*\w{\sU(\mathrm{pr}^*\T_X)}\cong\w{\D}_X$ for every $i\geq 0$. We now show that $\mathrm{R}^i\mathrm{pr}_*\N$ is naturally isomorphic to $\mathrm{R}^ij_*\M$.\\
\\
Let $\mathfrak{V}=(V_i)$ be a finite affinoid covering of $P$ where $V_i=W_i\times Y_i$, $W_i$ an admissible open affinoid subspace of $\mathbb{P}^{n, \mathrm{an}}$, $Y_i$ an affinoid subdomain of $X$, and $X'\cap V_i$ has a free tangent sheaf. Refining this covering, we can moreover assume that the complement of $U\cap V_i$ inside $X'\cap V_i$ is cut out by a single equation for each $i$.\\
By comparing the Cech complex $\check{C}^{\bullet}(\mathfrak{V}, \N)$ to $\check{C}^{\bullet}(U\cap \mathfrak{V}, \M)$ and invoking Corollary \ref{weaklyStein}, it suffices to show that we have natural isomorphisms
\begin{equation*}
\N(V)\cong \M(V\cap U)
\end{equation*}
where $V$ is any intersection of the $V_i$.\\
\\
Let $W$ be an admissible open affinoid subspace of $\mathbb{P}^{n, \mathrm{an}}$, $Y=\Sp C$ an affinoid subdomain of $X$ such that $X'\cap(W\times Y)$ has a free tangent sheaf, and write $V=W\times Y$.\\
Now
\begin{equation*}
\N(V)=\iota_+(j'_*\M)(V)\underset{\w{\D}(V)}{\h{\otimes}} \w{\sU(\mathrm{pr}^*\T_X)}(V),
\end{equation*}
and this is isomorphic to
\begin{equation*}
(j'_*\M)(X'\cap V)\underset{\w{\D}(X'\cap V)}{\h{\otimes}}\left(\O(X'\cap V)\underset{\O(V)}{\h{\otimes}} \w{\D}(V)\right)\underset{\w{\D}(V)}{\h{\otimes}}\left(\O(V)\underset{C}{\h{\otimes}} \w{\D}(Y)\right).
\end{equation*}
By associativity of the completed tensor product \cite[Lemma A.3]{Bitoun}, this can be simplified to
\begin{equation*}
(j'_*\M)(X'\cap V)\underset{\w{\D}(X'\cap V)}{\h{\otimes}}(\O(X'\cap V)\underset{C}{\h{\otimes}}\w{\D}(Y)),
\end{equation*}
which in turn can be written as
\begin{equation*}
\left(\M(U\cap V)\underset{\w{\D}_U(U\cap V)}{\h{\otimes}}(\O(U\cap V)\underset{\O(X'\cap V)}{\h{\otimes}}\w{\D}(X'\cap V)\right)\underset{\w{\D}(X'\cap V)}{\h{\otimes}}\left(\O(X'\cap V)\underset{C}{\h{\otimes}} \w{\D}(Y)\right),
\end{equation*}
since $\w{\D}(U\cap V)\cong \O(U\cap V)\underset{\O(X'\cap V)}{\h{\otimes}}\w{\D}_{X'}(X'\cap V)$ by Proposition \ref{ctp}.\\
\\
We thus obtain
\begin{align*}
\N(V)&\cong\M(U\cap V)\underset{\w{\D}_U(U\cap V)}{\h{\otimes}}\left(\O(U\cap V)\underset{C}{\h{\otimes}} \w{\D}_X(Y)\right)\\
&\cong \M(U\cap V)\underset{\w{\D}(U\cap V)}{\h{\otimes}} \left(\O(U\cap V)\underset{\O(U\cap Y)}{\h{\otimes}}(\O(U\cap Y)\underset{C}{\h{\otimes}} \w{\D}(Y)) \right)\\
&\cong \M(U\cap V)\underset{\w{\D}(U\cap V)}{\h{\otimes}} \left(\O(U\cap V)\underset{\O(U\cap Y)}{\h{\otimes}} \w{\D}(U\cap Y) \right)\\ 
& \cong \M(U\cap V)\underset{\w{\D}(U\cap V)}{\h{\otimes}} \w{\D}(U\cap V)\\
& \cong\M(U\cap V),
\end{align*}
by invoking once more Proposition \ref{ctp}, and \cite[Proposition 2.11.(ii)]{Bitoun}.
\end{proof}
\begin{cor}
Let $j:U\to X$ be a Zariski open embedding of smooth rigid analytic $K$-spaces, and let $\M$ be an integrable connection on $X$. Then $\mathrm{R}^ij_*(\M|_U)$ is a coadmissible $\w{\D_X}$-module for every $i\geq 0$.
\end{cor}
\begin{proof}
Without loss of generality, we can assume that $X$ is affinoid with free tangent sheaf. We therefore can consider the diagram (\ref{ERS}). \\
By Proposition \ref{snc}, $j'_*(\M|_U)$ is a coadmissible $\w{\D}_{X'}$-module. Now apply Theorem \ref{gencase}. 
\end{proof}

\subsection{Weak holonomicity}\label{pushfweakhol}
\begin{lem}
Let $X=\Sp A$ be a smooth affinoid $K$-space, and let $U=X_f$ for $f\in A$ non-constant. Let $\M$ be an integrable connection on $X$. Then $j_*(\M|_U)$ is a weakly holonomic $\w{\D}_X$-module, where $j: U \to X$ is the natural embedding.
\end{lem}
\begin{proof}
Without loss of generality, we can assume that $\T_X$ is free as an $\O_X$-module. By Corollary \ref{gencase}, $j_*(\M|_U)$ is a coadmissible $\w{\D_X}$-module, and it remains to show that the dimension of $\M(U)$ as a $\w{\D}_X(X)$-module is $\dim X$. We write $M=\M(X)$.\\
By \cite[Th\'eor\`eme 3.2.1]{Mebkhout}, $M[f^{-1}]$ is a $\D(X)$-module of minimal dimension, so that 
\begin{equation*}
d(\w{\D}(X)\otimes_{\D(X)} M[f^{-1}])=\dim X
\end{equation*} 
by Proposition \ref{extn}. But by \cite[Proposition 2.14]{Bitoun}, the natural morphism
\begin{equation*}
\theta: \w{\D}(X)\otimes_{\D(X)}M[f^{-1}]\to \M(U)
\end{equation*}
is a surjection of coadmissible $\w{\D}(X)$-modules, so $d_{\w{\D}(X)}(\M(U))=\dim X$ by Proposition \ref{weakhol}. 
\end{proof}
We can now give a proof of Theorem \ref{introthmB}.
\begin{thm}
Let $j: U\to X$ be a Zariski open embedding of smooth rigid analytic $K$-spaces, and let $\M$ be an integrable connection on $X$. Then $\mathrm{R}^ij_*(\M|_U)$ is a coadmissible, weakly holonomic $\w{\D}_X$-module for each $i\geq 0$.
\end{thm}
\begin{proof}
Without loss of generality, $X=\Sp A$ is affinoid with free tangent sheaf. There are finitely many elements $f_1, \dots, f_r\in A$ such that $U=X\setminus V(f_1, \dots, f_r)$. Note that the $X_{f_i}$ form an admissible covering of $U$, which we denote by $\mathfrak{V}$.\\
\\
By Corollary \ref{weaklyStein}, $\mathrm{R}^ij_*(\M|_U)(X)\cong \check{\mathrm{H}}^i(\mathfrak{V}, \M|_U)$. By Lemma \ref{pushfweakhol}, $\check{C}^\bullet(\mathfrak{V}, \M|_U)$ is a complex of coadmissible $\w{\D}_X(X)$-modules of minimal dimension, so $\mathrm{R}^ij_*(\M|_U)(X)$ is a coadmissible $\w{\D}_X(X)$-module of minimal dimension by Proposition \ref{weakhol}.
\end{proof}
\subsection{Local cohomology}\label{loccoh}
Let $X$ be a smooth rigid analytic $K$-space and let $Z$ be a closed analytic subset. The local cohomology sheaf functor $\underline{H}^i_Z(-)$ is then, as usual, the $i$th derived functor of $\underline{H}^0_Z(-)$, which assigns to a coherent $\O_X$-module its maximal subsheaf with support in $Z$ (see \cite[Definition 2.1.3]{Kisin99}).
\begin{thm}
Let $Z$ be a closed analytic subset of a smooth rigid analytic $K$-space $X$. Let $\M$ be an integrable connection on $X$. Then $\underline{H}^i_Z(\M)$ is a coadmissible, weakly holonomic $\w{\D}_X$-module for each $i\geq 0$.
\end{thm}
\begin{proof}
Suppose that $X$ is affinoid with free tangent sheaf. Let $j:U\to X$ be the complement of $Z$. As in \cite[Proposition 2.1.4]{Kisin99}, we can consider the exact sequence
\begin{equation*}
0 \to \underline{H}^0_Z(\M)\to \M\to j_*(\M|_U)\to \underline{H}^1_Z(\M)\to 0,
\end{equation*}
as well as the isomorphism $\underline{H}^i_Z(\M)\cong \mathrm{R}^{i-1}j_*(\M|_U)$ for any $i\geq 2$. Thus the result follows immediately from Theorem \ref{pushfweakhol} for $i\geq 2$, and from the fact that $\C_X^{\mathrm{wh}}$ is an abelian category for $i=0, 1$.
\end{proof}
\bibliography{references}
\bibliographystyle{plain}

\end{document}